\documentclass[a4paper,10pt,reqno]{amsart}
\usepackage{amsmath}
\usepackage{cases}
\usepackage{amsfonts}
\usepackage[colorlinks,linkcolor=blue,citecolor=blue]{hyperref}
\usepackage{latexsym, amssymb, amsmath, amsthm, bbm}
\usepackage[all]{xy}
\usepackage{pgfplots}
\usepackage{enumerate}
\usepackage{mathrsfs}

\DeclareSymbolFont{EulerExtension}{U}{euex}{m}{n}
\DeclareMathSymbol{\euintop}{\mathop} {EulerExtension}{"52}
\DeclareMathSymbol{\euointop}{\mathop} {EulerExtension}{"48}

\allowdisplaybreaks[4]

\def \id{\operatorname{id}}

\def \C{\mathcal{C}}

\def \Z{\mathbb{Z}}

\def \k{\Bbbk}

\def \dim{\operatorname{dim}}

\def \Ext{\operatorname{Ext}}

\def \C{\mathcal{C}}
\def \D{\Delta}

\def \span{\operatorname{span}}
\def \A{\mathcal{A}}
\def \B{\mathcal{B}}
\def \C{\mathcal{C}}
\def \D{\mathcal{D}}
\def \E{\mathcal{E}}
\def \F{\mathcal{F}}
\def \G{\mathcal{G}}
\def \X{\mathcal{X}}
\def \Y{\mathcal{Y}}
\def \Z{\mathcal{Z}}

\usepackage{mathtools}

\numberwithin{equation}{section}

\newtheorem{theorem}{Theorem}[section]
\newtheorem{lemma}[theorem]{Lemma}
\newtheorem{proposition}[theorem]{Proposition}
\newtheorem{corollary}[theorem]{Corollary}
\newtheorem{definition}[theorem]{Definition}
\newtheorem{example}[theorem]{Example}
\newtheorem{remark}[theorem]{Remark}

\begin{document}
\title[Hopf algebras of discrete corepresentation type]{Hopf algebras with the dual Chevalley property of discrete corepresentation type}
\thanks{}
\author[J. Yu]{Jing Yu}
\author[G. Liu]{Gongxiang Liu}
\address{School of Mathematical Sciences, University of Science and Technology of China, Hefei 230026, China}
\email{dg21210018@smail.nju.edu.cn}
\address{School of Mathematics, Nanjing University, Nanjing 210093, China}
\email{gxliu@nju.edu.cn}

\thanks{2020 \textit{Mathematics Subject Classification}. 16T05, 16G60 (primary), 16G20, 16G10 (secondary).}
\keywords{Hopf algebras, Dual Chevalley property, Discrete corepresentation type}
\maketitle
\begin{abstract}
We aim to classify Hopf algebras with the dual Chevalley property of discrete corepresentation type over an algebraically closed field $\k$ with characteristic 0. For such Hopf algebra $H$, we characterize the link quiver of $H$ and determine the structures of the
link-indecomposable component $H_{(1)}$ containing $\k1$. Besides, we construct an infinite-dimensional non-pointed non-cosemisimple link-indecomposable Hopf algebra $H(e_{\pm 1}, f_{\pm 1}, u, v)$ with the dual Chevalley property of discrete corepresentation type.
\end{abstract}
\date{}
\section{Introduction}
Inspired by the Drozd's result (\cite{Dro79}), one is often interested in classifying a given kind of finite-dimensional algebras according to their representation type. In the case of Hopf algebras, much effort was put in pointed Hopf algebras or their dual, that is, basic Hopf algebras. See, for example, \cite{BD82, Cil97, Far06, FS02, FS07, FV00, FV03, Hig54, Sut94, Xia97}. The second author and his collaborators gave a classification of basic Hopf algebras according to their representation type from 2006 to 2013 \cite{LL07,HL09, Liu06, Liu13}. They showed that a finite-dimensional basic Hopf algebra is of finite representation type if and only if it is a Nakayama algebra.

Meanwhile, Hopf algebras with the (dual) Chevalley property is a kind of natural generalization of basic (pointed) Hopf algebras. These Hopf algebras have been studied intensively by many authors. See, for examples, \cite{ABM12, AEG01, AGM17, Li22, LZ19}.
In \cite{YLL24} and \cite{YL24}, the authors tried to classify finite-dimensional Hopf algebras with the dual Chevalley property according to their corepresentation type. Here by the dual Chevalley property we mean that the coradical $H_0$ is a Hopf subalgebra. They proved that a finite-dimensional Hopf algebra $H$ with the dual Chevalley property is of finite corepresentation type if and only if it is coNakayama, if and only if the link quiver $\mathrm{Q}(H)$ of $H$ is a disjoint union of basic cycles, if and only if the link-indecomposable component $H_{(1)}$ containing $\k1$ is a pointed Hopf algebra and the link quiver of $H_{(1)}$ is a basic cycle.

One of the most important topics in representation theory is the classification of indecomposable modules over an algebra. The category of finite-dimensional left (right) modules over a finite representation type algebra is considered easiest to understand.
However, concerning infinite-dimensional Hopf algebras, it is no longer appropriate to discuss the (co)representation finiteness. Instead, we shall consider (co)representation discrete type (co)algebras as analogs of finite dimensional (co)representation finite type (co)algebras. We say a coalgebra $C$ is of discrete corepresentation type if for any finite dimension vector $\underline{d}$, there are only finitely many non-isomorphic indecomposable right $H$-comodules of dimension vector $\underline{d}$.
The authors in \cite{Iov18a} and \cite{ISSZ24} classified pointed Hopf algebras of discrete corepresentation type over an algebraically closed field $\k$ with characteristic zero. For such Hopf algebra $H$, they explicitly determined the algebra structure up to isomorphism for the link-indecomposable component $H_{(1)}$ containing $\k1$.

The aim of this paper is to classify Hopf algebras with the dual Chevalley property of discrete corepresentation type. The main tool we want to use is the link quiver. In fact, one can describe the structures of the link quiver by applying multiplicative matrices and primitive matrices (see \cite{YLL24}).
Denote the set of all the simple subcoalgebras of a Hopf algebra $H$ with the dual Chevalley property by $\mathcal{S}$. We can view the set $^{\C}{\mathcal{P}}^{\D}$ of a complete family of non-trivial $(\C, \D)$-primitive matrices as the set of arrows from vertex $D$ to vertex $C$. Denote
$
{}^{\C}\mathcal{P}=\bigcup_{\D\in\mathcal{S}}{}^{\C}{\mathcal{P}}^{\D},
{\mathcal{P}}^{\D}=\bigcup_{\C\in\mathcal{S}}{}^{\C}{\mathcal{P}}^{\D}, \mathcal{P}=\bigcup_{\C\in\mathcal{S}} {}^{\C}\mathcal{P}.
$
We can also view ${\mathcal{P}^{\D}}$ as the set of arrows with start vertex $D$ and view ${^{\C}\mathcal{P}}$ as the set of arrows with end vertex $C$. This means that we can view $\mathrm{Q}(H)=(\mathcal{S}, \mathcal{P})$ as the link quiver of $H$.

Based on the consideration above, we characterize the link quiver of a non-cosemisimple Hopf algebra $H$ with the dual Chevalley property of discrete corepresentation type. Our main results are Theorems \ref{Thm:H10finite} and \ref{thm:H10infinite}, stating that:
\begin{theorem}
Let $H$ be a non-cosemisimple Hopf algebra over $\k$ with the dual Chevalley property and $H_{(1)}$ be its link-indecomposable component containing $\k1$. If the coradical of $H_{(1)}$ is finite-dimensional, then the following statements are equivalent:
\begin{itemize}
  \item[(1)]$H$ is of discrete corepresentation type;
  \item[(2)]Every vertex in $\mathrm{Q}(H)$ is both the start vertex of only one arrow and the end vertex of only one arrow, that is, $\mathrm{Q}(H)$ is a disjoint union of basic cycles;
  \item[(3)]There is only one arrow $C\rightarrow \k1$ in $\mathrm{Q}(H)$ whose end vertex is $\k1$ and $\dim_{\k}(C)=1$;
  \item[(4)]There is only one arrow $\k1\rightarrow D$ in $\mathrm{Q}(H)$ whose start vertex is $\k1$ and $\dim_{\k}(D)=1$.
\end{itemize}
\end{theorem}

\begin{theorem}
Let $H$ be a non-cosemisimple Hopf algebra over $\k$ with the dual Chevalley property of discrete corepresentation type and $H_{(1)}$ be its link-indecomposable component containing $\k1$. Denote ${}^1\mathcal{S}=\{C\in\mathcal{S}\mid \k1+C\neq \k1\wedge C\}$. If the coradical of $H_{(1)}$ is infinite-dimensional, then one of the following three cases occurs:
  \begin{itemize}
  \item[(1)]$\mid{}^1\mathcal{P}\mid=1$ and ${}^1\mathcal{S}=\{\k g\}$ for some $g\in G(H)$;
  \item[(2)]$\mid{}^1\mathcal{P}\mid=2$ and ${}^1\mathcal{S}=\{\k g, \k h\}$ for some different group-like elements $g, h$;
  \item[(3)]$\mid{}^1\mathcal{P}\mid=1$ and ${}^1\mathcal{S}=\{C_k\}$ for some $C_k\in\mathcal{S}$ with $\dim_{\k}(C_k)=4$.
  \end{itemize}
\end{theorem}
In addition, for a non-cosemisimple Hopf algebra $H$ over $\k$ with the dual Chevalley property of discrete corepresentation type, we determine the structures of the
link-indecomposable component $H_{(1)}$ containing $\k1$.
According to \cite[Proposition 4.14]{YLL24}, if all the simple subcoalgebras directly linked to $\k 1$ are $1$-dimensional, then the link-indecomposable component $H_{(1)}$ containing $\k 1$ is a pointed Hopf algebra. Thus $H_{(1)}$ is a pointed Hopf algebra of discrete corepresentation type, which has been classified in \cite{Iov18a} and \cite{ISSZ24}.

Finally we deal with the remaining case, namely, $\mid{}^1\mathcal{P}\mid=1$ and ${}^1\mathcal{S}=\{C_k\}$, where $\dim_{\k}(C_k)=4$. We give a description of the structures of Grothendieck ring $\operatorname{Gr}((H_{(1)})_0$-comod$)$ of the category of finite-dimensional right $(H_{(1)})_0$-comodules and characterize the link quiver of $H_{(1)}$ in Subsection \ref{subsection4.3}. Besides, we construct an infinite-dimensional non-pointed non-cosemisimple link-indecomposable Hopf algebra $H(e_{\pm 1}, f_{\pm 1}, u, v)$ with the dual Chevalley property of discrete corepresentation type (see Definition \ref{def:Hefuv}).

The organization of this paper is as follows: In Section \ref{section1}, we recall some basic facts about based ring and link quiver. In Section \ref{section2}, we introduce the concept of discrete corepresentation type and formulate a necessary criterion for discreteness. In Section \ref{section3}, we characterize the link quiver of a non-cosemisimple Hopf algebra $H$ with the dual Chevalley property of discrete corepresentation type. Moreover, we determine the structures of the
link-indecomposable component $H_{(1)}$ containing $\k1$. At last, we construct an infinite-dimensional non-pointed non-cosemisimple link-indecomposable Hopf algebra $H(e_{\pm 1}, f_{\pm 1}, u, v)$ with the dual Chevalley property of discrete corepresentation type in Section \ref{section4}.
\section{Preliminaries}\label{section1}
Throughout this paper $\k$ denotes an \textit{algebraically closed field of characteristic $0$} and all spaces are over $\k$. The tensor product over $\k$ is denoted simply by $\otimes$.
The reader is referred to \cite{Mon93}, \cite{ARS95} and \cite{ASS06} for the basics about Hopf algebras and representation theory.
\subsection{Based ring}\label{subsection1.1}
Let us first recall the definition of multiplicative matrices.
\begin{definition}\emph{(}\cite[Definition 2.3]{Li22}\emph{)}
Let $(H,\Delta,\varepsilon)$ be a coalgebra over $\k$.
\begin{itemize}
  \item[(1)] A square matrix $\G=(g_{ij})_{r\times r}$ over $H$ is said to be multiplicative, if for any $1\leq i,j \leq r$, we have $\Delta(g_{ij})=\sum\limits_{t=1}^r g_{it}\otimes g_{tj}$ and $\varepsilon(g_{ij})=\delta_{i, j}$, where $\delta_{i, j}$ denotes the Kronecker notation;
  \item[(2)] A multiplicative matrix $\C$ is said to be basic, if its entries are linearly independent.
\end{itemize}
\end{definition}
Multiplicative matrices over a coalgebra can be understood as a generalization of group-like elements. The entries of a basic multiplicative matrix $\C$ span a simple subcoalgebra $C$ of $H$. Conversely, for any simple coalgebra $C$ over
$\k$, there exists a basic multiplicative matrix $\C$ whose entries span $C$ (for details, see \cite{LZ19}, \cite{Li22}).
According to \cite[Lemma 2.4]{Li22}, the basic multiplicative matrix of the simple coalgebra $C$ would be unique up to the similarity relation.

Let $\Bbb{Z}_+$ be the set of nonnegative integers. Some relevant concepts are recalled as follows.
\begin{definition}\emph{(}\cite[Definitions 2.1 and 2.2]{Ost03}\emph{)}\label{def:based}
Let $A$ be an associative ring with unit which is free as a $\Bbb{Z}$-module.
\begin{itemize}
  \item[(1)]A $\Bbb{Z}_+$-basis of $A$ is a basis $B=\{b_{i}\}_{i\in I}$ such that $b_ib_j=\sum_{t\in I}c_{ij}^tb_t$, where $c_{ij}^t\in\Bbb{Z}_+$.
  \item[(2)]A ring with a fixed $\Bbb{Z}_+$-basis $\{b_i\}_{i\in I}$ is called a unital based ring if the following conditions hold:
  \begin{itemize}
  \item[(i)]$1$ is a basis element.
  \item[(ii)]Let $\tau: A\rightarrow \Bbb{Z}$ denote the group homomorphism defined by
  $$\tau(b_i)=\left\{
\begin{aligned}
1,~~~  \text{if} ~~~ b_i=1, \\
0,~~~  \text{if} ~~~ b_i\neq1.
\end{aligned}
\right.$$
There exists an involution $i \mapsto i^*$ of $I$ such that the induced map
$$a=\sum\limits_{i\in I}a_ib_i \mapsto a^*=\sum\limits_{i\in I}a_ib_{i^*},\;\; a_i\in \Bbb{Z}$$ is an anti-involution of $A$, and
$$\tau(b_ib_j)=\left\{
\begin{aligned}
1,~~~  \text{if} ~~~ i=j^*, \\
0,~~~  \text{if} ~~~ i\neq j^*.
\end{aligned}
\right.$$
  \end{itemize}
\end{itemize}
\end{definition}
Recall that a finite-dimensional Hopf algebra is said to have the dual Chevalley property, if its coradical $H_0$ is a Hopf subalgebra. In this paper, we still use the term \textit{dual Chevalley property} to indicate a Hopf algebra $H$ with its coradical $H_0$ as a Hopf subalgebra, even if $H$ is infinite-dimensional.

In the following part, let $H$ be a Hopf algebra over $\k$ with the dual Chevalley property. Denote the coradical filtration of $H$ by $\{H_n\}_{n\geq0}$ and the set of all the simple subcoalgebras of $H$ by $\mathcal{S}$.

For any matrix $\A=(a_{ij})_{r\times s}$ and $\B=(b_{ij})_{u\times v}$ over $H$, define $\A\odot^\prime \B$ as follows:
 $$
\A\odot^\prime \B=\left(\begin{array}{cccc}
      \A b_{11} &   \cdots &  \A b_{1v} \\
      \vdots &  \ddots & \vdots  \\
      \A b_{u1} &  \cdots & \A b_{uv}
    \end{array}\right).$$

For any $B, C\in\mathcal{S}$, let $\mathcal{B}, \mathcal{C}$ be their basic multiplicative matrices, respectively.
Since $H$ has the dual Chevalley property, \cite[Proposition 2.6(2)]{Li22} implies there exists an invertible matrix $L$ over $\k$ such that
\begin{equation}\label{equationCD}
L
(\B\odot^{\prime}\C) L^{-1}=
\left(\begin{array}{cccc}
      \E_1 & 0 & \cdots & 0  \\
      0 & \E_2 & \cdots & 0  \\
      \vdots & \vdots & \ddots & \vdots  \\
      0 & 0 & \cdots & \E_t
    \end{array}\right),
    \end{equation}
where $\E_1, \E_2, \cdots, \E_t$ are basic multiplicative matrices over $H$.
Define a multiplication on $\Bbb{Z}\mathcal{S}$ as follows: for $B, C\in \mathcal{S}$,
$$B\cdot C=\sum\limits_{i=1}^t E_i,$$
where $E_1, \cdots, E_t\in\mathcal{S}$ are well-defined with basic multiplicative matrices $\E_i$ as in (\ref{equationCD}).
\begin{remark}
Observe that the equality $$B\cdot C=\sum\limits_{i=1}^t E_i$$ in $\Bbb{Z}\mathcal{S}$ implies $$\sqrt{\dim_{\Bbbk}(B)}\sqrt{\dim_{\Bbbk}(C)}=\sum\limits_{i=1}^t \sqrt{\dim_{\Bbbk}(E_i)}.$$
\end{remark}
Let $S$ be the antipode of $H$. By \cite[Theorem 3.3]{Lar71}, the map $C\mapsto S(C)$ defines an anti-involution. With the multiplication and anti-involution defined above, we state the following lemma.
\begin{lemma}\emph{(}\cite[Proposition 4.3]{YLL24}\emph{)}\label{Prop:basedring}
Let $H$ be a Hopf algebra over $\k$ with the dual Chevalley property and $\mathcal{S}$ be the set of all the simple subcoalgebras of $H$. Then $\Bbb{Z}\mathcal{S}$ is a unital based ring with $\Bbb{Z}_+$-basis $\mathcal{S}$.
\end{lemma}

\begin{remark}
By Definition \ref{def:based} (2) (ii), for any simple subcoalgebra $C$, the term $\k 1$ appears exactly once in the direct sum decomposition of $S(C) \cdot C$.
\end{remark}

Let $\mathcal{F}$ be the free abelian group generated by isomorphism classes of finite-dimensional right $H_0$-comodules and $\mathcal{F}_0$ the subgroup of $\mathcal{F}$ generated by all expressions $[Y]-[X]-[Z]$, where $0\rightarrow X\rightarrow Y\rightarrow Z\rightarrow0$ is a short exact sequence of finite-dimensional right $H_0$-comodules.
Recall that the \textit{Grothendieck group} $\operatorname{Gr}(H_0$-comod$)$ of the category of finite-dimensional right $H_0$-comodules is defined by $$\operatorname{Gr}(H_0\text{-comod}):=\mathcal{F}/\mathcal{F}_0.$$
From \cite[Proposition 4.5.4]{EGNO15} and \cite[Theorem 2.7]{Lar71}, $\operatorname{Gr}(H_0$-comod$)$ is a unital based ring with $\Bbb{Z}_+$-basis $\mathcal{V}$, where $\mathcal{V}$ is the set of all the isomorphism classes of simple right $H_0$-comodules.

Let $(M, \rho)$ be a finite-dimensional right comodule over a coalgebra $H^\prime$, where $\rho: M\rightarrow M\otimes H^\prime$. The coefficient coalgebra $\operatorname{cf}(M)$ of $M$ is the smallest subcoalgebra of $H^\prime$ satisfying $\rho(M)\subseteq M\otimes \operatorname{cf}(M)$. One can show that:
\begin{lemma}\label{lem:Grring}
Let $H$ be a Hopf algebra over $\k$ with the dual Chevalley property and $\mathcal{S}$ be the set of all the simple subcoalgebras of $H$. Then $\operatorname{Gr}(H_0$-comod$)$ is isomorphic to $\Bbb{Z}\mathcal{S}$ as unital based rings.
\end{lemma}
\begin{proof}
 Define
\begin{eqnarray*}
F:\operatorname{Gr}(H_0\text{-comod})&\rightarrow&\Bbb{Z}\mathcal{S},\\
 M\;\;\;\;&\mapsto &\operatorname{cf}(M).
\end{eqnarray*}
Next we show that $F$ is a ring isomorphism.
In fact, since $H_0$ is cosemisimple, it follows that $M$ is a completely irreducible right $H_0$-comodule. In other words, there are simple right $H_0$-comodules $V_1, V_2, \cdots, V_t$ such that $M=\bigoplus_{1\leq i\leq t}V_i$. Note that for any simple right $H_0$-comodule $V_i$, its coefficient coalgebra $\operatorname{cf}(V_i)$ is a simple subcoalgebra of $H$.
If $V_i$ and $V_j$ are non-isomorphic as right $H_0$-comodules, it is apparent that $\operatorname{cf}(V_i)$ and $\operatorname{cf}(V_j)$ are non-isomorphic as subcoalgebras. This means that $F$ is injective. Furthermore, for any $C\in\mathcal{S}$, any simple right $C$-comodule $X$ is a simple $H_0$-comodule and the coefficient coalgebra of $X$ is $C$. One can show that $F$ is surjective.
Using the fact that the coefficient coalgebra $\operatorname{cf}(V_i\otimes V_j)$ of $V_i\otimes V_j$ is $\operatorname{cf}(V_i)\operatorname{cf}(V_j)$, we get that $F$ is a ring isomorphism.
\end{proof}

\subsection{Link quiver}\label{subsection1.2}
In this subsection, let $(H,\Delta,\varepsilon)$ be a coalgebra over $\k$.
Denote the coradical filtration of $H$ by $\{H_n\}_{n\geq0}$ and the set of all the simple subcoalgebras of $H$ by $\mathcal{S}$.
Now let us recall the concept of link quiver.
\begin{definition}\emph{(}\cite[Definition 4.1]{CHZ06}\emph{)}
Let $H$ be a coalgebra over $\k$. Denote the set of all the simple subcoalgebras of $H$ by $\mathcal{S}$. The link quiver $\mathrm{Q}(H)$ of $H$ is defined as follows: the vertices of $\mathrm{Q}(H)$ are the elements of $\mathcal{S}$; for any simple subcoalgebra $C, D\in \mathcal{S}$ with $\dim_{\k}(C)=r^2, \dim_{\k}(D)=s^2$, there are exactly $\frac{1}{rs}\dim_{\k}((C\wedge D)/(C+D))$ arrows from $D$ to $C$.
\end{definition}
Next we will discuss the properties for the link quiver. Before proceeding further, let us recall the definition of primitive matrices, which is a non-pointed analogue of primitive elements.
\begin{definition}\emph{(}\cite[Definition 3.2]{LZ19}\emph{)}
Let $(H,\Delta,\varepsilon)$ be a coalgebra over $\k$. Suppose $\C=(c_{ij})_{r\times r}$ and $\D=(d_{ij})_{s\times s}$ are basic multiplicative matrices over $H$.
\begin{itemize}
  \item[(1)] A matrix $\X=(x_{ij})_{r\times s}$ over $H$ is said to be $(\C, \D)$-primitive, if $$\Delta(x_{ij})=\sum\limits_{k=1}^r c_{ik}\otimes x_{kj}+\sum\limits_{t=1}^s x_{it}\otimes d_{tj}$$ holds for any $1\leq i,j \leq r$;
  \item[(2)] A primitive matrix $\X$ is said to be non-trivial, if there exists some entry of $\X$ which does not belong to the coradical $H_0$.
\end{itemize}
\end{definition}
For any matrix $\X=\left(x_{ij}\right)_{r\times s}$ over $H_1$,  denote the matrix $\left(\overline{x_{ij}}\right)_{r\times s}$ by $\overline{\X}$, where $\overline{x_{ij}}=x_{ij}+H_0\in H_1/H_0$. Besides, the subspace of $H_1/H_0$ spanned by the entries of $\overline{\X}$ is denoted by $\operatorname{span}(\overline{\X})$.

Let $C, D\in\mathcal{S}$ with basic multiplicative matrices $\C$ and $\D$, respectively.
According to \cite[Corollary 2.11 and Lemma 2.17]{YLL24}, we know that there exists a family of non-trivial $(\C, \D)$-primitive matrices $\{\X^{(\gamma_{\C, \D})}\}_{\gamma_{\C, \D}\in\Gamma_{\C, \D}}$, which is said to be complete, such that
$(C\wedge D)/(C+D)\cong \bigoplus\limits_{\gamma_{\C, \D}\in\Gamma_{\C, \D}}\span(\overline{\X^{(\gamma_{\C, \D})}}).$ Using \cite[Corollary 2.18]{YLL24}, we can transform the problem of number of arrows from vertex $D$ to vertex $C$ in the link quiver of $H$ to the problem of cardinal number of a complete family of non-trivial $(\C, \D)$-primitive matrices.

Let ${}^{\C}{\mathcal{P}}^{\D}$ be the set of a complete family of non-trivial $(\C, \D)$-primitive matrices. Denote
\begin{eqnarray*}
{}^{\C}\mathcal{P}&=&\bigcup_{\D\in\mathcal{S}}{}^{\C}{\mathcal{P}}^{\D},\\
{\mathcal{P}}^{\D}&=&\bigcup_{\C\in\mathcal{S}}{}^{\C}{\mathcal{P}}^{\D},\\
\mathcal{P}&=&\bigcup_{\C\in\mathcal{S}} {}^{\C}\mathcal{P}.
\end{eqnarray*}
Now we can view $^{\C}{\mathcal{P}}^{\D}$ as the set of arrows from vertex $D$ to vertex $C$, view ${\mathcal{P}^{\D}}$ as the set of arrows with start vertex $D$ and view ${^{\C}\mathcal{P}}$ as the set of arrows with end vertex $C$. This means that we can view $\mathrm{Q}(H)=(\mathcal{S}, \mathcal{P})$ as the link quiver of $H$.

In the following part, let $H$ be a Hopf algebra over $\k$ with the dual Chevalley property.
In \cite[Section 3]{YLL24}, the authors gave two different constructions of a complete family of non-trivial $(\C, \D)$-primitive matrices over $H$ for any $C, D\in\mathcal{S}$ with basic multiplicative matrices $\C, \D$, respectively. Let us briefly recall one of the constructions.

Let $\mathcal{M}$ denote the set of representative elements of basic multiplicative matrices over $H$ for the similarity class.
Denote ${}^1\mathcal{S}=\{C\in\mathcal{S}\mid \k1+C\neq \k1\wedge C\}$. For any $C\in{}^1\mathcal{S}$ with basic multiplicative matrix $\C\in\mathcal{M}$, we can fix a complete family $\{\mathcal{X}^{(\gamma_{1, C})}\}_{\gamma_{1, \C}\in\Gamma_{1, \C}}$ of non-trivial $(1, \C)$-primitive matrices.

Denote
\begin{eqnarray*}
{^1\mathcal{P}}:=\bigcup\limits_{C\in {}^1\mathcal{S}} \{\mathcal{X}^{(\gamma_{1, \mathcal{C}})}\mid \gamma_{1, \mathcal{C}}\in\Gamma_{1, \mathcal{C}}\}.
\end{eqnarray*}
For any non-trivial $(1, \C)$-primitive matrix $\Y\in{^1\mathcal{P}}$ and $\B\in\mathcal{M}$, we have
\begin{eqnarray*}
\left(\begin{array}{cc}
I&0\\
0&L_{\B, \C}
 \end{array}\right)
\left(\B\odot^\prime
\left(\begin{array}{cc}
1&\Y\\
0&\C
 \end{array}\right)\right)
 \left(\begin{array}{cc}
I&0\\
0&L_{\B, \C}^{-1}
 \end{array}\right)
=\left(\begin{array}{ccccccc}
    \mathcal{B}  & {\mathcal{Y}_{ 1}} & {\mathcal{Y}_{ 2}} & \cdots & {\mathcal{Y}_{ u_{(\mathcal{B}, \mathcal{C})}}}  \\
    0 &   \mathcal{E}_{1} &0&\cdots  &  0 \\
    0& 0&\mathcal{E}_{2}&\cdots &0\\
    \vdots  &\vdots  &\vdots& \ddots & \vdots  \\
    0   & 0 &0  &\cdots& \mathcal{E}_{u_{(\mathcal{B}, \mathcal{C})}}
  \end{array}\right),
\end{eqnarray*}
where $L_{\B, \C}$ is an invertible matrix over $\k$ and $\E_1, \E_2, \cdots, \E_{u_{(\B, \C)}}\in \mathcal{M}$.

Denote
\begin{eqnarray*}
^{\B}\mathcal{P}_{\Y}:=\{\Y_{ i}\mid 1\leq i\leq u_{(\B, \C)}\},
\end{eqnarray*}
\begin{eqnarray*}
 \mathcal{P}_{\Y}:=\bigcup\limits_{\B\in \mathcal{M}}{}^{\B}\mathcal{P}_{\Y}.
\end{eqnarray*}

As a consequence, we obtain the following:
\begin{lemma}\emph{(}\cite[Theorem 3.10]{YLL24}\emph{)}\label{coro:BXcomplete}
Let $H$ be a Hopf algebra over $\k$ with the dual Chevalley property and $C, D\in \mathcal{S}$ with basic multiplicative matrices $\C, \D\in\mathcal{M}$ respectively. Then the set
$$\{\mathcal{X}\in \bigcup\limits_{\mathcal{Y}\in{{}^1\mathcal{P}}}\mathcal{P}_{\mathcal{Y}} \mid  \mathcal{X} \text{ is a non-trivial }(\mathcal{C}, \mathcal{D})\text{-primitive matrix}\}$$
is a complete family of non-trivial $(\mathcal{C}, \mathcal{D})$-primitive matrices. Moreover, we have $$H_1/H_0=\bigoplus_{\mathcal{X}\in\bigcup\limits_{\mathcal{Y}\in{{}^1\mathcal{P}}}\mathcal{P}_{\mathcal{Y}}}\operatorname{span}(\overline{\mathcal{X}}).$$
\end{lemma}

Let $C, D\in\mathcal{S}$ with basic multiplicative matrices $\C, \D\in\mathcal{M}$ respectively.
Recall that $C$, $D$ are said to be \textit{directly linked} in $H$ if $C+D$ is a proper subspace of $C\wedge D+D\wedge C$.
Note that by \cite[Lemma 3.6 (2)]{Li22} that $C$, $D$ are directly linked in $H$ if and only if there exists some $(\C, \D)$-primitive or $(\D, \C)$-primitive matrix, which is non-trivial.
At the end of this subsection, we recall the concept of link-indecomposable components of coalgebra.
\begin{definition}\emph{(}\cite[Definition 1.1]{Mon95}\emph{)}
A subcoalgebra $H^\prime$ of a coalgebra $H$ is called \textit{link-indecomposable} if the link quiver $\mathcal{Q}(H^\prime)$ of $H^\prime$ is connected (as an undirected graph).
A \textit{link-indecomposable component} of $H$ is a maximal link-indecomposable subcoalgebra. In particular, for a Hopf algebra $H$, we denote the link-indecomposable component containing $\k 1$ by $H_{(1)}$.
\end{definition}
According to \cite[Theorem 3.2]{Mon95}, the link-indecomposable component $H_{(1)}$ containing $\k1$ must be a normal Hopf subalgebra for any pointed Hopf algebra $H$. By \cite[Proposition 3.16]{Li22}, if $H$ is a Hopf algebra with the dual Chevalley property, $H_{(1)}$ remains a Hopf subalgebra of $H$. However, concerning a non-pointed Hopf algebra $H$ with the dual Chevalley property, $H_{(1)}$ may not necessarily be normal (see for example \cite[Example 6.1]{YLL24}). This demonstrates that Hopf algebras with the dual Chevalley property constitute a nontrivial generalization of pointed Hopf algebras.

\section{discrete corepresentation type}\label{section2}
Recall that the dimension vector of a finite dimensional right comodule $M$ over coalgebra $H$ is defined as $\underline{dim}(M)\in\Bbb{N}^{(I)}$ given by letting $\underline{dim}(M)_i$ equal the multiplicity of simple right comodule $V_i$ in a Jordan-H$\ddot{\textrm{o}}$lder series of $M$. In \cite[Definition 1.1]{ISSZ24}, the author introduced the definition of discrete corepresentation type for pointed coalgebra. We formally extend this concept.
\begin{definition}
A coalgebra $H$ is said to be of discrete corepresentation type, if for any finite dimension vector $\underline{d}$, there are only finitely many non-isomorphic indecomposable right $H$-comodules of dimension vector $\underline{d}$.
\end{definition}
It is clear that cosemisimple coalgebras are of discrete corepresentation type. Moreover, if a coalgebra $H$ is of finite corepresentation type, that is, there are finitely many non-isomorphic indecomposable right $H$-comodules, then $H$ is of discrete corepresentation type. Next we will focus on the general cases.

Let $H$ be a coalgebra and $\mathrm{Q}(H)=(\mathcal{S}, \mathcal{P})$ the link quiver of $H$. Suppose $\mathrm{Q}^\prime=(\mathcal{S}^\prime, \mathcal{P}^\prime)$ is a sub-quiver of
$\mathrm{Q}(H)$, where $\mathcal{S}^\prime\subseteq \mathcal{S}, \mathcal{P}^\prime\subseteq \mathcal{P}$. Define $\operatorname{Coalg}(\mathrm{Q}^\prime)=(\bigoplus_{\C\in\mathcal{S}^\prime}C)\oplus(\bigoplus_{\X\in\mathcal{P}^\prime}\span(\X))$, where $\span(\X)$ is the subspace of $H$ spanned by the entries of $\X$. It is straightforward to show that $\operatorname{Coalg}(\mathrm{Q}^\prime)$ is a subcoalgebra of $H$. This leads to the following Lemma.
\begin{lemma}\label{lem:subinfinite}
Let $H$ be a coalgebra over $\k$ and $\mathrm{Q}(H)=(\mathcal{S}, \mathcal{P})$ the link quiver of $H$. Suppose there is a finite sub-quiver $\mathrm{Q}^\prime=(\mathcal{S}^\prime, \mathcal{P}^\prime)$ of $\mathrm{Q}(H)$ satisfying the subcoalgebra $Coalg(\mathrm{Q}^\prime)$ is of infinite corepresentation type. Then $H$ is not of discrete corepresentation type.
\end{lemma}
\begin{proof}
Since $\mathrm{Q}(H)$ is the link quiver of $H$, it follows that $\operatorname{Coalg}(\mathrm{Q}^\prime)\subseteq \operatorname{Coalg}(\mathrm{Q(H)})\subseteq H$ are extensions of subcoalgebras. This means that there is an inclusion from the category of finite-dimensional right $\operatorname{Coalg}(\mathrm{Q}^\prime)$-comodules to the category of finite-dimensional right $H$-comodules. Note that $\textrm{Coalg}(\mathrm{Q}^\prime)$ is a finite-dimensional coalgebra of infinite corepresentation type and the category of finite-dimensional right comodules over $\textrm{Coalg}(\mathrm{Q}^\prime)$ is isomorphic to the category of finite-dimensional left modules over $(\textrm{Coalg}(\mathrm{Q}^\prime))^*$. It follows that $(\textrm{Coalg}(\mathrm{Q}^\prime))^*$ is a finite-dimensional algebra of infinite representation type. By \cite[Theorem 2.4]{Bau85}, there is an infinitely family of isomorphism classes of indecomposable right comodules over $\operatorname{Coalg}(\mathrm{Q}^\prime)$ with a dimension vector $\underline{d}$. Therefore the category of finite-dimensional right $H$-comodules contains infinitely many isomorphism classes of indecomposable comodules with a dimension vector $\underline{d}$, hence not of discrete corepresentation type.
\end{proof}
Let $A$ (resp. $C$) be an algebra (resp. coalgebra) over $\k$ and $\{M_i\}_{ i\in I}$ be the complete set of isoclasses of simple left $A$-modules (resp. right $C$-comodules). The \textit{Ext quiver} $\Gamma(A)$ (resp. $\Gamma(C)$) of $A$ (resp. $C$) is an oriented graph with vertices indexed by $I$, and there are $\dim_{\k}\Ext^1(M_i, M_j)$ arrows from $i$ to $j$ for any $i, j\in I$.

Let us recall the definition of separated quiver.
\begin{definition}\emph{(}cf. \cite[\textsection X. 2]{ARS95}\emph{)}
Let $\mathrm{Q}=(\mathrm{Q}_0, \mathrm{Q}_1)$ be a quiver, where $\mathrm{Q}_0=\{1, 2, \cdots, n\}$. The separated quiver $\mathrm{Q}_{s}$ of $\mathrm{Q}$ has $2n$ vertices $\{1, 2, \cdots, n, 1^\prime, 2^\prime, \cdots, n^\prime\}$ and an arrow $i \rightarrow j^\prime$ for every arrow $i \rightarrow j$ of $\mathrm{Q}$.
\end{definition}
Now we can formulate a necessary criterion for discreteness.
\begin{lemma}\label{lem:Dynkin}
Let $H$ be a non-cosemisimple coalgebra over $\k$ of discrete corepresentation type and $\mathrm{Q}(H)=(\mathcal{S}, \mathcal{P})$ the link quiver of $H$. Then for any finite sub-quiver $\mathrm{Q}^\prime$ of $\mathrm{Q}(H)$, the underlying graph of separated
quiver $\mathrm{Q}^\prime_s$ is a finite disjoint union of Dynkin diagrams.
\end{lemma}
\begin{proof}
In fact, the category of finite-dimensional right comodules over $Coalg(\mathrm{Q}^\prime)$ is isomorphic to the category of finite-dimensional left modules over $(Coalg(\mathrm{Q}^\prime))^*$. This means that the coalgebra's version of Ext quiver $\Gamma^{\mathrm{c}}$ of $Coalg(\mathrm{Q}^\prime)$ is the same as the algebra's version of Ext quiver $\Gamma^{\mathrm{a}}$ of $(Coalg(\mathrm{Q}^\prime))^*$. According to \cite[Theorem 2.1 and Corollary 4.4]{CHZ06}, the link quiver $\mathrm{Q}^\prime$ of $Coalg(\mathrm{Q}^\prime)$ coincides with the algebra's version of Ext quiver $\Gamma^{\mathrm{a}}$ of $(Coalg(\mathrm{Q}^\prime))^*$. Note that $(Coalg(\mathrm{Q}^\prime))^*$ is Morita equivalent to a basic algebra $B$.
Let $J$ be the ideal generated by all the arrows in $\mathrm{Q}^\prime$. By the Gabriel's theorem, there exists an admissible ideal $I$ such that $$\k\mathrm{Q}^\prime/I\cong B,$$ where $J^t\subseteq I \subseteq J^2$ for some integer $t\geq2$. Thus there exists an algebra epimorphism $$f:B\rightarrow \k\mathrm{Q}^\prime/J^2.$$  Since the Jacobson radical of $\k\mathrm{Q}^\prime/J^2$ is $J/J^2$, we know that $\k\mathrm{Q}^\prime/J^2$ is an artinian algebra with radical square zero.
It follows from the proof of \cite[X.2 Theorem 2.6]{ARS95} that the separated quiver of $\k\mathrm{Q}^\prime/J^2$ coincides with the quiver of the hereditary algebra
$\sum=\left(\begin{array}{cc}
    (\k\mathrm{Q}^\prime/J^2)/(J/J^2)   & 0 \\
    J/J^2  & (\k\mathrm{Q}^\prime/J^2)/(J/J^2)
  \end{array}\right).$
Note that $\k\mathrm{Q}^\prime/J^2$ and $\sum$ are stably equivalent, it follows that $\k\mathrm{Q}^\prime/J^2$ is of infinite representation type if and only if $\sum$ is of infinite representation type. Suppose the underlying graph of $\mathrm{Q}^\prime_s$ is not a finite disjoint union of Dynkin diagrams, then $\sum$ is of infinite representation type, which indicates that $\k\mathrm{Q}^\prime/J^2$ is of infinite representation type. Thus $B$ is of infinite representation type.
It follows that $Coalg(\mathrm{Q}^\prime)$ is of infinite corepresentation type.
According to Lemma \ref{lem:subinfinite}, we know that $H$ is not of discrete corepresentation type, which is a contradiction.
\end{proof}
Recall that a quiver $\mathrm{Q}$ is said to be Schurian, if for each pair $(C, D)$ of vertices of $\mathrm{Q}$, there are at most one arrow from $C$ to $D$. The following result is not hard:
\begin{corollary}\label{coro:schurian}
Let $H$ be a coalgebra over $\k$. If $H$ is of discrete corepresentation, then the link quiver $\mathrm{Q}(H)$ of $H$ is Schurian.
\end{corollary}
\begin{proof}
Otherwise, there exists some finite sub-quiver $\mathrm{Q}^\prime$ of $\mathrm{Q}(H)$ such that $\mathrm{Q}^\prime_s$ contains a Kronecker quiver as a sub-quiver. This is contrary to Lemma \ref{lem:Dynkin}.
\end{proof}

\section{Hopf algebras with the dual Chevalley property of discrete corepresentation type}\label{section3}
In this section, we classify Hopf algebras with the dual Chevalley property of discrete corepresentation type.

Let $H$ be a non-cosemisimple Hopf algebra over $\k$ with the dual Chevalley property and $\mathrm{Q}(H)$ the link quiver of $H$. Denote the coradical filtration of $H$ by $\{H_n\}_{n\geq0}$.
For convenience, denote $\mathcal{S}=\{C_i\mid i\in I\}$ the set of all the simple subcoalgebras of $H$. For any $C_i, C_j\in\mathcal{S}$, let $C_i\cdot C_j=\sum\limits_{t\in I}\alpha_{ij}^tC_t$ in $\Bbb{Z}\mathcal{S}$, where $\alpha_{ij}^t\in\Bbb{Z}_+$.
Moreover, we denote $\mathcal{M}=\{\C_j\mid i\in I\}$, such that each $\C_j\in\mathcal{M}$ is the basic multiplicative matrix of $C_j\in\mathcal{S}$.

Denote ${}^1\mathcal{S}=\{C\in\mathcal{S}\mid \k1+C\neq\k1\wedge C\}$, $\mathcal{S}^1=\{C\in\mathcal{S}\mid C+\k1\neq C\wedge \k1\}$. Observe that for any $C\in{}^1\mathcal{S}$, there exists some arrow from $C$ to $\k 1$ in the link quiver $\mathrm{Q}(H)$ of $H$. For any $C\in\mathcal{S}{}^1$, there exists some arrow from $\k 1$ to $C$ in the link quiver $\mathrm{Q}(H)$ of $H$.

The authors of \cite{YLL24} establish certain properties for the link quiver of a finite-dimensional Hopf algebra $H$ with the dual Chevalley property in Sections 4 and 5. As it turns out, several of these properties admit natural extensions to the infinite-dimensional setting. We now list some of these results from \cite{YLL24}, the proofs of which are omitted.

\begin{lemma}\label{lemma:P^1=1^P}\emph{(}\cite[Lemmas 4.6, 4.7 and 5.4, Propositions 4.9, 4.14, Corollary 4.10]{YLL24}\emph{)}
Let $H$ be a non-cosemisimple Hopf algebra over $\k$ with the dual Chevalley property.
\begin{itemize}
  \item[(1)]
  \begin{itemize}
  \item[(i)]We have $\mid{^1\mathcal{P}}\mid=\mid\mathcal{P}^1\mid\geq1$. Moreover, $C\in{}^1\mathcal{S}$ if and only if $S(C)\in\mathcal{S}^1$;
     \item[(ii)] $\mid{}^{\C}\mathcal{P}\mid=\mid\mathcal{P}^{\C}\mid=1 $ holds for all $\C\in\mathcal{M}$ if and only if
  $\mid{}^1\mathcal{P}\mid=1$ and the unique subcoalgebra $C\in{}^1\mathcal{S}$ is $1$-dimensional.
     \end{itemize}
  \item[(2)]For any $\Y\in{}^{1}\mathcal{P}$, where $\Y$ is a non-trivial $(1, \C_j)$-primitive matrix and $\C_j\in\mathcal{M}$, then the cardinal number $\mid{}^{\C_i}\mathcal{P}_{\Y}\mid=\sum\limits_{t\in I}\alpha_{ij}^t\geq1$.
  \item[(3)]If all the simple subcoalgebras directly linked to $\k1$ are $1$-dimensional, then we have
  \begin{itemize}
 \item[(i)]$\mid{^{\C}\mathcal{P}}\mid=\mid{\mathcal{P}^{\C}}\mid=\mid{^1\mathcal{P}}\mid$, for any $\C\in \mathcal{M}$;
 \item[(ii)]$H_{(1)}$ is a pointed Hopf algebra.
\end{itemize}
\item[(4)]
 \begin{itemize}
 \item[(i)]$\alpha_{ik}^t=\alpha_{tk^*}^i$ holds for any $i, j, k\in I$;
     \item[(ii)] If $\mid{}^1\mathcal{P}\mid=1$ and $C_k$ is the unique simple subcoalgebra contained in ${}^1\mathcal{S}$. Then the number of arrows with end vertex $C_i$ in $\mathrm{Q}(H)$ is equal to $\sum\limits_{t\in I}\alpha_{ik}^t$, and the number of arrows with start vertex $C_i$ in $\mathrm{Q}(H)$ is equal to $\sum\limits_{t\in I}\alpha_{ik^*}^t$. In particular, the number of arrows from $C_t$ to $C_i$ in $\mathrm{Q}(H)$ is equal to $\alpha_{ik}^t$.
     \end{itemize}
  \end{itemize}
\end{lemma}

Recall that a \textit{basic cycle} of length $n$ is a quiver with $n$ vertices $e_0, e_1, \cdots, e_{n-1}$ and $n$ arrows $a_0, a_1, \cdots a_{n-1}$, where the arrow $a_i$ goes from the vertex $e_i$ to the vertex $e_{i+1}$. A finite-dimensional algebra is said to be \textit{Nakayama}, if each indecomposable projective left and right module has a unique composition series. It is well-known that a finite-dimensional basic algebra $A$ is Nakayama if and only if every vertex of the Ext quiver of $A$ is the start vertex of at most one arrow and the end vertex of at most one arrow. With the help of the preceding lemma, we can now prove:
\begin{theorem}\label{Thm:H10finite}
Let $H$ be a non-cosemisimple Hopf algebra over $\k$ with the dual Chevalley property and $H_{(1)}$ be its link-indecomposable component containing $\k1$. If the coradical of $H_{(1)}$ is finite-dimensional, then the following statements are equivalent:
\begin{itemize}
  \item[(1)]$H$ is of discrete corepresentation type;
  \item[(2)]Every vertex in $\mathrm{Q}(H)$ is both the start vertex of only one arrow and the end vertex of only one arrow, that is, $\mathrm{Q}(H)$ is a disjoint union of basic cycles;
  \item[(3)]There is only one arrow $C\rightarrow \k1$ in $\mathrm{Q}(H)$ whose end vertex is $\k1$ and $\dim_{\k}(C)=1$;
  \item[(4)]There is only one arrow $\k1\rightarrow D$ in $\mathrm{Q}(H)$ whose start vertex is $\k1$ and $\dim_{\k}(D)=1$.
\end{itemize}
\end{theorem}
\begin{proof}
According to Lemma \ref{lemma:P^1=1^P} (1), we know the equivalence of (2), (3), and (4). It remains to show the equivalence of (1) and (2).

 Suppose $H$ is of discrete corepresentation type. Because of the fact that there is an inclusion from the category of finite-dimensional right $H_{(1)}$-comodules to the category of
finite-dimensional right $H$-comodules, we know that $H_{(1)}$ is of discrete corepresentation type. Clearly, if $(H_{(1)})_0$ is finite-dimensional, the number of simple subcoalgbras of $H_{(1)}$ is finite and \cite[Lemma 4.12]{YLL24} works. Then by Lemma \ref{lem:Dynkin} and the same reason in the proof of \cite[Theorem 5.6]{YLL24}, we know that $\mathrm{Q}(H_{(1)})$ is a basic cycle. According to Lemma \ref{lemma:P^1=1^P} (1), $\mathrm{Q}(H)$ is a disjoint union of basic cycles.

 Conversely, from the proof of Lemma \ref{lem:Dynkin}, for any finite-dimensional subcoalgebra $H^\prime$ of $H$, we know that the link quiver $\mathrm{Q}(H^\prime)$ of $H^\prime$ is the same as the Ext quiver $\Gamma(H^{\prime*})^\mathrm{a}$ of $H^{\prime*}$. Observe that $H^{\prime*}$ is Morita equivalent to a basic algebra $B(H^{\prime*})$.
Since every vertex in $\Gamma(H^{\prime*})^\mathrm{a}$ is the start vertex of at most one arrow and the end vertex of at most one arrow, it follows that the basic algebra $B(H^{\prime*})$ is a Nakayama algebra. By \cite[\textsection VI. Theorem 2.1]{ARS95}, $B(H^{\prime*})$ is of finite representation type, which implies that $H^\prime$ is of finite corepresentation type. For any finite dimension vector $\underline{d}$, denote by $\operatorname{cf}(\underline{d})$ the smallest subcoalgebra of $H$ such that all the right $H$-comodules of dimension vector $\underline{d}$ have their coefficient coalgebra contained in $\operatorname{cf}(\underline{d})$, that is, $$\operatorname{cf}(\underline{d})=\sum\limits_{\underline{dim}(M)=\underline{d}}\operatorname{cf}(M).$$
Using \cite[Lemma 2.6]{Iov18b}, we can show that $\operatorname{cf}(\underline{d})$ is finite-dimensional. It follows that $\operatorname{cf}(\underline{d})$ is of finite corepresentation type. Therefore, the set of isomorphism classes of dimension vector $\underline{d}$ corepresentations is finite. This implies that $H$ is of discrete corepresentation type.
\end{proof}
As a consequence of \cite[Theorem 5.6]{YLL24} and Theorem \ref{Thm:H10finite}, we have
\begin{corollary}
A finite-dimensional Hopf algebra $H$ over $\k$ with the dual Chevalley property is of finite corepresentation type if and only if it is of discrete corepresentation type.
\end{corollary}

Let $q\in \k$ be an $n$-th root of unit of order $d$. In \cite{Rad75} and \cite{AS98}, Radford and Andruskiewitsch-Schneider have considered the following Hopf algebra $A(n, d, \mu, q)$ which as an associative algebra is generated by $g$ and $x$ with relations
$$
g^n=1, \;\;\;\;x^d=\mu(1-g^d),\;\;\;\;xg=qgx.
$$

Its comultiplication $\Delta$, counit $\varepsilon$, and the antipode $S$ are given by
$$
\Delta(g)=g\otimes g,\;\; \varepsilon(g)=1,\;\;
\Delta(x)=1\otimes x+x\otimes g,\;\; \varepsilon(x)=0,\;\;
S(g)=g^{-1},\;\;S(x)=-xg^{-1}.
$$

According to \cite[Theorem 4.6]{LL07}, we know that $A(n, d, \mu, q)$ is a finite-dimensional link-indecomposable pointed Hopf algebra of finite corepresentation type.
\begin{corollary}\label{coro:=kx}
Let $H$ be a non-cosemisimple Hopf algebra over $\k$ with the dual Chevalley property and $H_{(1)}$ be its link-indecomposable component containing $\k1$. If the coradical of $H_{(1)}$ is finite-dimensional, then $H$ is of discrete corepresentation type if and only if $H_{(1)}$ is isomorphic to $A(n, d, \mu, q)$ or $\k[x]$.
\end{corollary}
\begin{proof}
If $H_{(1)}$ is isomorphic to $ A(n, d, \mu, q)$ or $\k[x]$, there is only one arrow $C\rightarrow \k1$ in $\mathrm{Q}(H)$ whose end vertex is $\k1$ and $\dim_{\k}(C)=1$. Using Theorem \ref{Thm:H10finite}, we know that $H$ is of discrete corepresentation type. Conversely, since $H$ is of discrete corepresentation type, it follows from Theorem \ref{Thm:H10finite} that the link quiver $\mathrm{Q}(H_{(1)})$ of $H_{(1)}$ is a basic cycle. According to Lemma \ref{lemma:P^1=1^P} (1), $H_{(1)}$ is a pointed Hopf algebra of discrete corepresentation type. It is a consequence of \cite[Section 6]{ISSZ24} that $H_{(1)}$ is isomorphic to $ A(n, d, \mu, q)$ or $\k[x]$.
\end{proof}

\begin{example}
Let $H$ be the Hopf algebra generated by $z, y, t, u$ satisfying the following relations:
$$z^{2}=1,\;\; y^{2}=1,\;\; t^{2}=1,\;\;  z y=y z,\;\;  t z=z t,\;\; t y=y t,$$
$$zu=uz, \;\;yu=uy,\;\;tu=ut.$$

The coalgebra structure and antipode are given by:
$$\Delta(z)=z \otimes z,\;\; \Delta(y)=y \otimes y,\;\; \varepsilon(z)=\varepsilon(y)=1,$$
$$\Delta(t)=\frac{1}{2}\left[(1+y) t \otimes t+(1-y) t \otimes z t\right],\;\;  \varepsilon(t)=1,$$
$$\Delta(u)=1\otimes u+u\otimes 1,\;\;\varepsilon(u)=0,$$
$$S(z)=z,\;\;  S(y)=y,\;\;  S(t)=\frac{1}{2}\left[(1+y) t+(1-y) z t\right],\;\; S(u)=-u.$$

Denote $E=\operatorname{span}\{t, zt, yt, zyt\}$, then $\mathcal{S}=\{\k 1, \k  z, \k  y, \k  zy,  E\}$. We give the corresponding multiplicative matrix $\E$ of $E$, where
$$
\E=\frac{1}{2}\left(\begin{array}{cc}
t+yt&t-yt\\
zt-zyt&zt+zyt
 \end{array}\right).
$$

By the definition of the comultiplication of $H$, we know that $(u)$ is a non-trivial $((1), (1))$-primitive matrix, that is, $u$ is a non-trivial primitive element. Moreover, we have ${}^1\mathcal{P}=\{(u)\}$. Let $\mathcal{M}$ be a set of representative elements of basic multiplicative matrices over $H$ for the similarity class. For any $\mathcal{B}\in\mathcal{M}$, we have
$$
\mathcal{B}\odot^\prime
\left(\begin{array}{cc}
1&u\\
0&1
 \end{array}\right)=
 \left(\begin{array}{cc}
\mathcal{B}&\mathcal{B}u\\
0&\mathcal{B}
 \end{array}\right).
$$
According to Lemma \ref{coro:BXcomplete}, we know that $\mathcal{P}=\{(u), (zu), (yu), (zyu), \mathcal{X}\}$, where $$
\mathcal{X}=\frac{1}{2}\left(\begin{array}{cc}
(t+yt)u&(t-yt)u\\
(zt-zyt)u&(zt+zyt)u
 \end{array}\right).
$$
It follows that $(zu)$ is a non-trivial $( (z),  (z))$-primitive matrix, $(yu)$ is a non-trivial $((y), (y))$-primitive matrix, $(zyu)$ is a non-trivial $((zy), (zy))$-primitive matrix and $\mathcal{X}$ is a non-trivial $(\mathcal{E}, \mathcal{E})$-primitive matrix.
Thus the link quiver of $H$ is shown below:
$$
\begin{tikzpicture}
\filldraw [black] (0,0) circle (0.5pt) node[anchor=west]{$\k  1$};
\filldraw [black] (2,0) circle (0.5pt)node[anchor=west]{$\k  z$};
\filldraw [black] (4,0) circle (0.5pt) node[anchor=west]{$\k y$};
\filldraw [black] (6,0) circle (0.5pt)node[anchor=west]{$\k  zy$};
\filldraw [black] (8,0) circle (0.5pt)node[anchor=west]{$E$};
\draw[thick, ->] (0,0.1) arc (0:340:0.6);
\draw[thick, ->] (2,0.1) arc (0:340:0.6);
\draw[thick, ->] (4,0.1) arc (0:340:0.6);
\draw[thick, ->] (6,0.1) arc (0:340:0.6);
\draw[thick, ->] (8,0.1) arc (0:340:0.6);
\end{tikzpicture}
$$
From Theorem \ref{Thm:H10finite}, $H$ is a non-pointed Hopf algebra with the dual Chevalley property of discrete corepresentation type.
\end{example}
Next we consider the case when $(H_{(1)})_0$ is infinite-dimensional. Before proceeding further, let us give the following lemma.
\begin{lemma}\label{lem:C>4}
Let $H$ be a non-cosemisimple Hopf algebra over $\k$ with the dual Chevalley property. If there exists some $C_k\in {}^1\mathcal{S}$ such that $\dim_{\k}(C_k)\geq 9$, then $H$ is not of discrete corepresentation type.
\end{lemma}
\begin{proof}
According to Lemma \ref{lem:Dynkin}, the key idea of the proof is to find a finite sub-quiver $\mathrm{Q}^\prime$ of $\mathrm{Q}(H)$ such that the
underlying graph of separated quiver $\mathrm{Q}^\prime_s$ is not a disjoint union of Dynkin diagrams.
Using Lemma \ref{lemma:P^1=1^P} (4) (i), we know that the following two numbers are equal:
\begin{itemize}
\item[-]The number of $C_t$ contained in $C_i\cdot C_k$;
\item[-]The number of $C_i$ contained in $C_t\cdot S(C_k)$.
\end{itemize}
To prove this lemma, we divide the argument into several cases.
\begin{itemize}
  \item[(I)]Suppose $$S(C_k)\cdot C_k=\sum_{i\in I}\alpha_{k^*k}^i C_i$$ in $\Bbb{Z}\mathcal{S}$, where $\sum_{i\in I}\alpha_{k^*k}^i\geq 4$. Then by Lemma \ref{lemma:P^1=1^P} (2), the separated quiver of $\mathrm{Q}(H)$ contains a vertex which is the end vertex of at least $4$ arrows. Evidently, we can find a finite sub-quiver $\mathrm{Q}^\prime$ of $\mathrm{Q}(H)$ such that the
underlying graph of separated quiver $\mathrm{Q}^\prime_s$ is not a disjoint union of Dynkin diagrams.
  \item[(II)]Suppose $$S(C_k)\cdot C_k=\k1+D_1+D_2$$ in $\Bbb{Z}\mathcal{S}$,  where $\sqrt{\dim_{\k}(D_1)}\geq \sqrt{\dim_{\k}(D_2)}.$ Since $\sqrt{\dim_{\k}(C_k)}\geq 3$, it follows that $$\sqrt{\dim_{\k}(D_1)}+\sqrt{\dim_{\k}(D_2)}\geq 8$$ and $$\sqrt{\dim_{\k}(D_1)}>\sqrt{\dim_{\k}(S(C_k))}.$$
    From Lemma \ref{lemma:P^1=1^P} (4), we have
    $$D_1\cdot S(C_k)=S(C_k)+\sum_{i\in I}\beta_iC_i$$ in $\Bbb{Z}\mathcal{S}$, where $\beta_i\in\Bbb{Z}_+$ for any $i\in I$ and $\sum_{i\in I}\beta_i\geq 1$.
    In fact, if $\sum_{i\in I}\beta_i\geq 2$, by Lemma \ref{lemma:P^1=1^P} (2), there exists a finite sub-quiver $\mathrm{Q}^\prime$ of $\mathcal{Q}(H)$ such that $\mathrm{Q}^\prime_s$ contains at least one vertex which is the end vertex of 3 arrows and at least one vertex which is the start vertex of 3 arrows. That is, $\mathrm{Q}_s$ contains either
$$
\begin{tikzpicture}
\filldraw [black] (1,0) circle (0.5pt) node[anchor=south]{};
\filldraw [black] (-1,0) circle (0.5pt)node[anchor=north]{};
\draw[thick, ->] (-0.9,0.1).. controls (-0.1,0.1) and (0.1,0.1) .. node[anchor=south]{}(0.9,0.1) ;
\draw[thick, ->] (-0.9,-0.1).. controls (-0.1,-0.1) and (0.1,-0.1) .. node[anchor=south]{}(0.9,-0.1);
\end{tikzpicture},
$$
or
    $$
\begin{tikzpicture}
\filldraw [black] (0,0) circle (0.5pt) node[anchor=south]{$\k1$};
\filldraw [black] (1,0) circle (0.5pt)node[anchor=south]{$D_2$};
\filldraw [black] (2,0) circle (0.5pt)node[anchor=south]{$D_1$};
\filldraw [black] (1,-1) circle (0.5pt) node[anchor=north]{$S(C_k)^\prime$};
\filldraw [black] (3,-1) circle (0.5pt)node[anchor=north]{};
\filldraw [black] (4,-1) circle (0.5pt)node[anchor=north]{};
\draw[thick, ->] (0.1,-0.1).. controls (0.1,-0.1) and (0.1,-0.1) .. node[anchor=south]{}(0.9,-0.9) ;
\draw[thick, ->] (1,-0.1).. controls (1,-0.1) and (1,-0.1) .. node[anchor=south]{}(1,-0.9) ;
\draw[thick, ->] (1.9,-0.1).. controls (1.9,-0.1) and (1.9,-0.1) .. node[anchor=south]{}(1.1,-0.9) ;
\draw[thick, ->] (2.1,-0.1).. controls (2.1,-0.1) and (2.1,-0.1) .. node[anchor=south]{}(2.9,-0.9) ;
\draw[thick, ->] (2.2,-0.1).. controls (2.2,-0.1) and (2.2,-0.1) .. node[anchor=south]{}(3.9,-0.9) ;
\end{tikzpicture}
$$
as a sub-quiver.
The underlying graph of the sub-quiver in the latter case is $\tilde{D_5}$ and it is an Euclidean
graph.
As a result, $H$ is not of discrete corepresentation type. Now suppose $\sum_{i\in I}\beta_i=1$, which means that
     $$D_1\cdot S(C_k)=S(C_k) +E_1$$  in $\Bbb{Z}\mathcal{S}$. Clearly, we have
     \begin{eqnarray*}
     \sqrt{\dim_{\k}(E_1)}&=&\sqrt{\dim_{\k}(S(C_k))}\sqrt{\dim_{\k}(D_1)}-\sqrt{\dim_{\k}(S(C_k))}\\
     &>&\sqrt{\dim_{\k}(S(C_k))}\sqrt{\dim_{\k}(D_1)}-\sqrt{\dim_{\k}(D_1)}\\
     &>& \sqrt{\dim_{\k}(D_1)}.
     \end{eqnarray*}
     It follows from Lemma \ref{lemma:P^1=1^P} (4) that $$E_1\cdot C_k=D_1 +\sum_{i\in I}\gamma_iC_i$$ in $\Bbb{Z}\mathcal{S}$, where $\gamma_i\in\Bbb{Z}_+$ for any $i\in I$ and $\sum_{i\in I}\gamma_i\geq 1$. A similar argument shows that either there exists a finite sub-quiver $\mathrm{Q}^\prime$ of $\mathrm{Q}(H)$ such that $\mathrm{Q}^\prime_s$ contains at least one vertex which is the end vertex of 3 arrows and at least one vertex which is the start vertex of 3 arrows or $\sum_{i\in I}\gamma_i= 1$. In the later case, we have $$E_1\cdot C_k=D_1 +E_2$$ in $\Bbb{Z}\mathcal{S}$ and
     $\sqrt{\dim_{\k}(E_2)}>  \sqrt{\dim_{\k}(E_1)}$. Continue the steps, if $H$ is not of discrete corepresentation type, we can get an infinite sequence
     $$E_i\cdot S^{i+1}(C)=E_{i-1}+E_{i+1}$$
     such that $\sqrt{\dim_{\k}(E_{i+1})}> \sqrt{\dim_{\k}(E_i)}$, where $i\geq 1$ and $E_{0}=D_1$. One can finally get an infinite-dimensional simple subcoalgebra, which is impossible. Based on the above argument, we know that $H$ is not of discrete corepresentation type in this case.
      \item[(III)]Finally, we focus on the case that $$S(C_k)\cdot C_k=\k1+D_1$$ in $\Bbb{Z}\mathcal{S}$, where $\sqrt{\dim_{\k}(D_1)}> \sqrt{\dim_{\k}(S(C_k))}$.
      \begin{itemize}
      \item[(i)]If $$D_1\cdot S(C_k)=S(C_k)+\sum\limits_{i\in I}\beta_i C_i$$in $\Bbb{Z}\mathcal{S}$, where $\beta_i\in\Bbb{Z}_+$ for any $i\in I$ and $\sum\limits_{i\in I}\beta_i\geq 3$. Using the same argument as in (I), we can easily show that $H$ is not of discrete corepresentation type.
       \item[(ii)]If $$D_1\cdot S(C_k)=S(C_k)+D_2+D_3$$in $\Bbb{Z}\mathcal{S}$, where $\sqrt{\dim_{\k}(D_2)}\geq \sqrt{\dim_{\k}(D_3)}$. We have
        \begin{eqnarray*}
        \sqrt{\dim_{\k}(D_2)}&\geq&\frac{1}{2}(\sqrt{\dim_{\k}(D_1)}\sqrt{\dim_{\k}(S(C_k))}-\sqrt{\dim_{\k}(S(C_k))})\\
        &>&\frac{1}{2}(\sqrt{\dim_{\k}(D_1)}\sqrt{\dim_{\k}(S(C_k))}-\sqrt{\dim_{\k}(D_1)})\\
        &\geq&\sqrt{\dim_{\k}(D_1)}.
        \end{eqnarray*}
        It follows that $$D_2\cdot C_k=D_1 +\sum\limits_{i\in I}\gamma_i C_i$$in $\Bbb{Z}\mathcal{S}$, where $\gamma_i\in \Bbb{Z}_+$ and $\sum\limits_{i\in I}\gamma_i\geq 1.$ If $\sum\limits_{i\in I}\gamma_i\geq 2,$
        by Lemma \ref{lemma:P^1=1^P} (2), there exists a finite sub-quiver $\mathrm{Q}^\prime$ of $\mathcal{Q}(H)$ such that $\mathrm{Q}^\prime_s$ contains at least one vertex which is the end vertex of 3 arrows and at least one vertex which is the start vertex of 3 arrows. This means that $H$ is not of discrete corepresentation type.  If $\sum\limits_{i\in I}\gamma_i=1,$ that is, $$D_2\cdot C_k=D_1 +D_4$$ in $\Bbb{Z}\mathcal{S}$, where $\sqrt{\dim_{\k}(D_4)}> \sqrt{\dim_{\k}(D_2)}$. Continue the steps, an argument similar to the one used in (II) shows that $H$ is not of discrete corepresentation type.
       \item[(iii)]If $$D_1\cdot S(C_k)=S(C_k)+D_2$$in $\Bbb{Z}\mathcal{S}$, where $\sqrt{\dim_{\k}(D_2)}> \sqrt{\dim_{\k}(D_1)}$. By adopting the same procedure as in (III)(i) and (ii), we can show that either $H$ is not of discrete corepresentation type or we can get an infinite sequence
     $$D_i\cdot S^{i}(C)=D_{i-1}+D_{i+1}$$
     such that $\sqrt{\dim_{\k}(D_{i+1})}> \sqrt{\dim_{\k}(D_i)}$, where $i\geq 1$ and $D_{0}=S(C_k)$. In the later case, one can finally get an infinite-dimensional simple subcoalgebra, which is impossible.
      \end{itemize}
     In conclusion, $H$ is not of discrete corepresentation type.
  \end{itemize}
\end{proof}
Now we can characterize the link quiver of $H$ when $H$ is of discrete corepresentation type and $(H_{(1)})_0$ is infinite-dimensional.
\begin{theorem}\label{thm:H10infinite}
Let $H$ be a non-cosemisimple Hopf algebra over $\k$ with the dual Chevalley property of discrete corepresentation type and $H_{(1)}$ be its link-indecomposable component containing $\k1$. Denote ${}^1\mathcal{S}=\{C\in\mathcal{S}\mid \k1+C\neq \k1\wedge C\}$. If the coradical of $H_{(1)}$ is infinite-dimensional, then one of the following three cases occurs:
  \begin{itemize}
  \item[(1)]$\mid{}^1\mathcal{P}\mid=1$ and ${}^1\mathcal{S}=\{\k g\}$ for some $g\in G(H)$;
  \item[(2)]$\mid{}^1\mathcal{P}\mid=2$ and ${}^1\mathcal{S}=\{\k g, \k h\}$ for some different group-like elements $g, h\in G(H)$;
  \item[(3)]$\mid{}^1\mathcal{P}\mid=1$ and ${}^1\mathcal{S}=\{C_k\}$ for some $C_k\in\mathcal{S}$ with $\dim_{\k}(C_k)=4$.
  \end{itemize}
\end{theorem}
\begin{proof}
For any $n\geq 2$, let $\mathcal{S}(n)$ be the set of all the $n^2$-dimensional simple subcoalgebras of $H$.
\begin{itemize}
  \item[(I)]
If $\mid{}^1\mathcal{P}\mid\geq 3$, by adopting the same procedure as in the proof of \cite[Theorem 4.2]{YL24}, one can find a finite sub-quiver $\mathrm{Q}^\prime$ of $\mathrm{Q}(H)$ such that the separated quiver $\mathrm{Q}^\prime_s$ of $\mathrm{Q}^\prime$ is not a disjoint union of Dynkin diagrams. Using Lemma \ref{lem:Dynkin}, we know that $H$ is not of discrete corepresentation type.
\item[(II)]
According to Lemma \ref{lem:C>4}, if there exists some $C_k\in {}^1\mathcal{S}$ such that $\dim_{\k}(C_k)\geq 9$, then $H$ is not of discrete corepresentation type.
\item[(III)]
Suppose that $${}^1\mathcal{P}=\{\X, \Y\},$$ where $\X$ is a non-trivial $(1, \k g)$-primitive matrix and $\Y$ is a non-trivial $(1, \C_k)$-primitive matrix for some $g\in G(H)$ and $C_k\in\mathcal{S}(2)$. Note that $H$ is of discrete corepresentation type and ${}^1\mathcal{S}$ contains a $4$-dimensional simple subcoalgebra. Proceeding as in the proof of \cite[Proposition 5.5]{YLL24}, we can get an infinite sequence $(a)$
\begin{eqnarray*}
S(C_k)\cdot C_k&=&\k1+\k g+D_1^{(2)},\\
D_1^{(2)}\cdot S(C_k)&=&S(C_k)+D_2^{(2)},\\
D_2^{(2)}\cdot C_k&=&D_1^{(2)}+D_3^{(2)},\\
&\cdots&\\
D_{2i}^{(2)}\cdot C_k&=&D_{2i-1}^{(2)}+D_{2i+1}^{(2)},\\
D_{2i+1}^{(2)}\cdot S(C_k)&=&D_{2i}^{(2)}+D_{2i+2}^{(2)},\\
&\cdots&
\end{eqnarray*}
or an infinite sequence $(b)$
\begin{eqnarray*}
S(C_k)\cdot C_k&=&\k1+D_1^{(3)},\\
D_1^{(3)}\cdot S(C_k)&=&S(C_k)+D_1^{(4)},\\
D_1^{(4)}\cdot C_k&=&D_1^{(3)}+D_1^{(5)},\\
&\cdots&\\
D_1^{(2i)}\cdot C_k&=&D_1^{(2i-1)}+D_1^{(2i+1)},\\
D_1^{(2i+1)}\cdot S(C_k)&=&D_1^{(2i)}+D_1^{(2i+2)},\\
&\cdots&,
\end{eqnarray*}
where $D_i^{(j)}\in \mathcal{S}(j), g\in G(H)$. Otherwise, we can find a finite sub-quiver $\mathrm{Q}^\prime$ of $\mathrm{Q}(H)$ such that the separated quiver $\mathrm{Q}^\prime_s$ of $\mathrm{Q}^\prime$ is not a disjoint union of Dynkin diagrams, which is a contradiction to Lemma \ref{lem:Dynkin}. Next we deal with case $(a)$ and $(b)$.
In case $(a)$, according to \cite[Corollary 3.9]{YLL24} and Lemma \ref{lemma:P^1=1^P} (2), we have $$\mid{}^{KS(\C_k)K^{-1}}\mathcal{P}\mid=\mid{}^{KS(\C_k)K^{-1}}\mathcal{P}_{\X}\mid+\mid{}^{KS(\C_k)K^{-1}}
\mathcal{P}_{\Y}\mid=4,$$
where $K$ is an invertible matrix over $\k$ such that $KS(\C_k)K^{-1}\in\mathcal{M}$ is the basic multiplicative matrix of $S(C_k)$.
Thus $\mathrm{Q}(H)_s$ contains at least one vertex which is the end vertex of at least 4 arrows. In such a case, we can show that $H$ is not of discrete corepresentation type by Lemma \ref{lem:Dynkin}.
In case $(b)$, one can get an infinite-dimensional simple subcoalgebra, which leads to a contradiction.
\item[(IV)]
By Corollary \ref{coro:schurian}, we know that if $\mid{}^1\mathcal{P}\mid=2$ and ${}^1\mathcal{S}=\{\k g, \k h\}$ for some $g, h\in G(H),$ then $g\neq h.$
\end{itemize}
To conclude, if $H$ is of discrete corepresentation type, then either $\mid{}^1\mathcal{P}\mid=1$ and ${}^1\mathcal{S}=\{C_k\}$ for some $C_k\in\mathcal{S}$ with $\dim_\k(C_k)\leq 4$, or $\mid{}^1\mathcal{P}\mid=2$ and ${}^1\mathcal{S}=\{\k g, \k h\}$ for some different group-like elements $g, h\in G(H)$.
\end{proof}

In the following part, let $H$ be a non-cosemisimple Hopf algebra over $\k$ with the dual Chevalley property of discrete corepresentation type such that the coradical of $H_{(1)}$ is infinite-dimensional.
Next we give an accurate description for $H_{(1)}$ when $H$ is of discrete corepresentation type. We discuss these three cases separately.
\subsection{Case (1)}
Suppose $\mid{}^1\mathcal{P}\mid=1$ and ${}^1\mathcal{S}=\{\k g\}$, where $g\in G(H)$. According to Lemma \ref{lemma:P^1=1^P} (3), we know that $$\mid{}^{\C}\mathcal{P}\mid=\mid\mathcal{P}^{\C}\mid=1.$$
It follows that the link quiver of $H_{(1)}$ is double infinite quiver ${}_{\infty}\mathcal{A}_{\infty}$:
$$
\begin{tikzpicture}
\filldraw [black] (0,0) circle (0.5pt) node[anchor=south]{};
\filldraw [black] (1,0) circle (0.5pt) node[anchor=south]{};
\filldraw [black] (-1,0) circle (0.5pt)node[anchor=north]{};
\filldraw [black] (-2,0) circle (0.5pt)node[anchor=east]{};
\filldraw [black] (-3,0) circle (0pt)node[anchor=east]{$\cdots$};
\filldraw [black] (2,0) circle (0pt)node[anchor=west]{$\cdots$};
\draw[thick, ->] (-0.9,0).. controls (-0.9,0) and (-0.9,0) .. node[anchor=south]{}(-0.1,0) ;
\draw[thick, ->] (0.1,0).. controls (0.1,0) and (0.1,0) .. node[anchor=south]{}(0.9,0) ;
\draw[thick, ->] (1.1,0).. controls (1.1,0) and (1.1,0) .. node[anchor=south]{}(1.9,0) ;
\draw[thick, ->] (-2.9,0).. controls (-2.9,0) and (-2.9,0) .. node[anchor=south]{}(-2.1,0) ;
\draw[thick, ->] (-1.9,0).. controls (-1.9,0) and (-1.9,0) .. node[anchor=south]{}(-1.1,0) ;
\end{tikzpicture}.
$$

Let us give several examples of pointed Hopf algebras whose link quiver is ${}_{\infty}\mathcal{A}_{\infty}$.
\begin{example}\emph{(}\cite[Section 2]{Liu09}\emph{)}
The Hopf algebra $A(n, q)$ is generated by $g, g^{-1}, x$ subjects to relations:
$$
gg^{-1}=g^{-1}g=1,\;\;xg=qgx,\;\;x^n=1-g^n,
$$
where $q\in\k$ is a $n$-th primitive root of unity.
The coalgebra structure and antipode are given by:
$$
\Delta(g)=g\otimes g,\;\; \Delta(x)=x\otimes 1+g\otimes x,
$$
$$
\varepsilon(g)=1,\;\;\varepsilon(x)=0,\;\; S(g)=g^{-1},\;\; S(x)=-g^{-1}x.
$$
\end{example}
\begin{example}\emph{(}\cite[Theorem 5.4]{Iov18a}\emph{)}
The Hopf algebra $H_{\infty}(\chi, \lambda)$ is generated by $g, g^{-1}, x$ subjects to relations:
$$
gg^{-1}=g^{-1}g=1,\;\;xg=\chi(g)gx+\lambda(g)(g-g^2),
$$
where $\chi$ is a $1$-dimensional character such that $\chi(g)=1$ or $\chi(g)$ is not a root of unity and $\lambda\in(\k \langle g, g^{-1}\rangle)^\circ$ is an element in the finite dual Hopf algebra such that $\lambda(hf)=\chi(h)\lambda(f)+\lambda(h)$ for any $f, h\in\k \langle g, g^{-1}\rangle.$
The coalgebra structure and antipode are given by:
$$
\Delta(g)=g\otimes g,\;\; \Delta(x)=1\otimes x+x\otimes g,
$$
$$
\varepsilon(g)=1,\;\;\varepsilon(x)=0,\;\; S(g)=g^{-1},\;\; S(x)=-xg^{-1}.
$$
\end{example}
Now we give a description of the algebra structure of $H_{(1)}$ in this case.
\begin{proposition}\label{prop:=Hinfi}
Let $H$ be a non-cosemisimple Hopf algebra over $\k$ with the dual Chevalley property such that the coradical of $H_{(1)}$ is infinite-dimensional. If $\mid{}^1\mathcal{P}\mid=1$ and ${}^1\mathcal{S}=\{\k g\}$, where $g\in G(H)$, then $H$ is of discrete corepresentation type. Moreover, $H_{(1)}$ is isomorphic to $A(n, q)$ or $H_{\infty}(\chi, \lambda)$.
\end{proposition}
\begin{proof}
Since $\mid{}^1\mathcal{P}\mid=1$ and ${}^1\mathcal{S}=\{\k g\}$, where $g\in G(H)$, it follows from Lemma \ref{lemma:P^1=1^P} (3) that
$$\mid{}^{\C}\mathcal{P}\mid=\mid\mathcal{P}^{\C}\mid=1.$$
Using the same argument as in the proof of Theorem $\ref{Thm:H10finite}$, we can easily show that $H$ is of discrete corepresentation type. According to Lemma \ref{lemma:P^1=1^P} (3), $H_{(1)}$ is a pointed Hopf algebra of discrete corepresentation type. Using \cite[Section 6]{ISSZ24}, we know that $H_{(1)}$ is isomorphic to $A(n, q)$ or $H_{\infty}(\chi, \lambda)$.
\end{proof}
\subsection{Case (2)}
Suppose $\mid{}^1\mathcal{P}\mid=2$ and ${}^1\mathcal{S}=\{\k g, \k h\}$, where $g, h$ are two different group-like elements.
It follows from Lemma \ref{lemma:P^1=1^P} (3) that $H_{(1)}$ is a pointed Hopf algebra of discrete corepresentation type.

Let $(m, n)$ be a pair of integers such that $(m, n)\neq \pm (1, 1)$.
Let $\mathrm{Q}^{m, n}$ be a quiver defined as follows: The set of vertices of $\mathrm{Q}^{m, n}_0=\langle g, h\mid gh=hg, g^m=h^n\rangle$, for each vertex, there is a unique arrow from $g^ih^j$ to $g^{i+1}h^j$ and a unique arrow from $g^ih^j$ to $g^{i}h^{j+1}$.
According to \cite[Theorem 4.9]{ISSZ24}, we know that the link quiver of $H_{(1)}$ in this case is $\mathrm{Q}^{m, n}.$

In \cite[Section 5]{ISSZ24}, the authors construct a family of Hopf algebra $B^{m, n}(\lambda, s, t, k)$ such that the link quiver of $B^{m, n}(\lambda, s, t, k)$ is $\mathrm{Q}^{m, n}.$
\begin{example}\emph{(}\cite[Definition 5.11]{ISSZ24}\emph{)}
For pairs of integers $(m, n)\neq \pm(1, 1)$ and $m+n\in2\Bbb{Z},$
define $B^{m, n}(\lambda, s, t, k)$ be the Hopf algebra generated by $g, h, x, y$ satisfying the following conditions, where $\lambda\neq0, s, t, k\in\k.$
$$
gh=hg,\;\;g^m=h^n,\;\; xy+\lambda yx=k(1-gh),
$$
$$
gx+xg=0,\;\;\lambda hx+xh=0,\;\;x^2=s(1-g^2),
$$
$$
hy+yh=0,\;\; gy+\lambda yg=0,\;\;y^2=t(1-h^2),
$$
$$
\Delta(g)=g\otimes g,\;\;\Delta(h)=h\otimes h,\;\; \Delta(x)=1\otimes x+x\otimes g\;\;\Delta(y)=1\otimes y+y\otimes h,
$$
$$
\varepsilon(g)=\varepsilon(h)=1,\;\;\varepsilon(x)=\varepsilon(y)=0,
$$
$$
S(g)=g^{-1}, \;\;S(h)=h^{-1},\;\;S(x)=-xg^{-1},\;\;S(y)=-yh^{-1}.
$$
\end{example}
Using \cite[Theorem 5.15]{ISSZ24}, we know that the Hopf algebra $B^{m, n}(\lambda, s, t, k)$ is of discrete corepresentation type if and only if $m\neq n$ or $m=n=0.$ Besides, the authors classified the Hopf algebra $B^{m, n}(\lambda, s, t, k)$ up to isomorphism in \cite[Lemma 5.16]{ISSZ24} and showed that there are more constraints on the parameters $\lambda, s, t, k$ in \cite[Lemma 5.18]{ISSZ24}.
As mentioned above, we can now obtain the following proposition.
\begin{proposition}
Let $H$ be a non-cosemisimple Hopf algebra over $\k$ with the dual Chevalley property of discrete corepresentation type such that the coradical of $H_{(1)}$ is infinite-dimensional. If $\mid{}^1\mathcal{P}\mid=2$ and ${}^1\mathcal{S}=\{\k g, \k h\}$, where $g, h$ are two different group-like elements, then $H_{(1)}$ is isomorphic to exactly one of the following:
\begin{itemize}
  \item[(1)]$B^{m, n}(\lambda, 0, 0, 0)$, $\lambda\in\k^\times$ and $\lambda^{gcd(m, n)=1}$ if $(m, n)\neq (0, 0)$;
  \item[(2)]$B^{m, n}(-1, 0, 1, 0)$, if both $m, n$ are even;
  \item[(3)]$B^{m, n}(-1, 1, 0, 0)$, if both $m, n$ are even;
  \item[(4)]$B^{m, n}(-1, 1, 1, 0)$, if both $m, n$ are even;
  \item[(5)]$B^{m, n}(1, 1, 1, k)$, $k\in\k$;
  \item[(6)]$B^{m, n}(1, 1, 0, 0)$;
  \item[(6')]$B^{m, n}(1, 0, 1, 0)$;
  \item[(7)]$B^{m, n}(1, 1, 0, 1)$;
  \item[(7')]$B^{m, n}(1, 0, 1, 0)$;
  \item[(8)]$B^{m, n}(1, 0, 0, 1)$.
\end{itemize}
\end{proposition}
\begin{proof}
It follows from Lemma \ref{lemma:P^1=1^P} (3) that $H_{(1)}$ is a pointed Hopf algebra of discrete corepresentation type. According to \cite[Section 6]{ISSZ24}, the proof is done.
\end{proof}
\subsection{Case (3)}\label{subsection4.3}
In this subsection, let $\mathcal{S}^\prime$ be the set of simple subcoalgebras of $H_{(1)}$ and $\operatorname{Gr}((H_{(1)})_0$-comod$)$ be the Grothendieck ring of the category of finite-dimensional right $(H_{(1)})_0$-comodules. Next we give a description of the structures of $\operatorname{Gr}((H_{(1)})_0$-comod$)$ and determine the link quiver of $H_{(1)}$ in the case that $\mid{}^1\mathcal{P}\mid=1$ and ${}^1\mathcal{S}=\{C_k\}$, where $\dim_{\k}(C_k)=4$.

It follows from \cite[Proposition 3.16]{Li22} that $H_{(1)}$ is a Hopf subalgebra.
Using Lemma \ref{lem:Grring}, we know that $\operatorname{Gr}((H_{(1)})_0$-comod$)$ is isomorphic to $\Bbb{Z}\mathcal{S}^\prime.$ Thus we only need to focus on the structures of $\Bbb{Z}\mathcal{S}^\prime.$
Before that, we give a new construction of a complete family of non-trivial $(\C, \D)$-primitive matrices over $H$ for any $C, D\in\mathcal{S}$ with basic multiplicative matrices $\C, \D$, respectively.

For any matrix $\A=(a_{ij})_{r\times s}$ and $\B=(b_{ij})_{u\times v}$ over $H$, define $\A\odot \B$ as follow
 $$\A\odot \B=
\left(\begin{array}{ccc}
      a_{11}\B& \cdots &  a_{1s}\B  \\
      \vdots  & \ddots & \vdots  \\
      a_{r1}\B&  \cdots & a_{rs}\B
    \end{array}\right).$$

For any $B, C\in\mathcal{S}$ with basic multiplicative matrices $\B, \C$ respectively.
Since $H$ has the dual Chevalley property, it follows from
\cite[Proposition 2.6(2)]{Li22} that there exists an invertible matrix $L$ over $\k$ such that
\begin{equation*}
L
(\C\odot\B) L^{-1}=
\left(\begin{array}{cccc}
      \F_1 & 0 & \cdots & 0  \\
      0 & \F_2 & \cdots & 0  \\
      \vdots & \vdots & \ddots & \vdots  \\
      0 & 0 & \cdots & \F_{u_{(\C, \B)}}
    \end{array}\right),
    \end{equation*}
where $\F_1, \F_2, \cdots, \F_t$ are basic multiplicative matrices over $H$.

Let $\mathcal{M}$ denote the set of representative elements of basic multiplicative matrices over $H$ for the similarity class.
For any $C\in{}^1\mathcal{S}$ with basic multiplicative matrix $\C\in\mathcal{M}$, we can fix a complete family $\{\mathcal{X}^{(\gamma_{1, C})}\}_{\gamma_{1, \C}\in\Gamma_{1, \C}}$ of non-trivial $(1, \C)$-primitive matrices.

Denote
\begin{eqnarray*}
{^1\mathcal{P}}:=\bigcup\limits_{C\in {}^1\mathcal{S}} \{\mathcal{X}^{(\gamma_{1, \mathcal{C}})}\mid \gamma_{1, \mathcal{C}}\in\Gamma_{1, \mathcal{C}}\}.
\end{eqnarray*}

Then for any non-trivial $(1, \C)$-primitive matrix $\Y\in{^1\mathcal{P}}$ and $\B\in\mathcal{M}$, we have
\begin{eqnarray*}
\left(\begin{array}{cc}
I&0\\
0&L
 \end{array}\right)
\left(
\left(\begin{array}{cc}
1&\Y\\
0&\C
 \end{array}\right)\odot \B\right)
 \left(\begin{array}{cc}
I&0\\
0&L^{-1}
 \end{array}\right)
=\left(\begin{array}{ccccccc}
    \mathcal{B}  & {\mathcal{Y}^\prime_{ 1}} & {\mathcal{Y}^\prime_{ 2}} & \cdots & {\mathcal{Y}^\prime_{ u_{(\mathcal{C}, \mathcal{B})}}}  \\
    0 &   \mathcal{F}_{1} &0&\cdots  &  0 \\
    0& 0&\mathcal{F}_{2}&\cdots &0\\
    \vdots  &\vdots  &\vdots& \ddots & \vdots  \\
    0   & 0 &0  &\cdots& \mathcal{F}_{u_{(\mathcal{C}, \mathcal{B})}}
  \end{array}\right),
\end{eqnarray*}
where $\F_1, \F_2, \cdots, \F_{u_{(\C, \B)}}\in \mathcal{M}$.

Denote
\begin{eqnarray*}
^{\B}\mathcal{P}^{\prime}_{\Y}:=\{\Y^\prime_{ i}\mid 1\leq i\leq u_{(\C, \B)}\},
\end{eqnarray*}
\begin{eqnarray*}
^{\B}\mathcal{P}^{\prime}:=\bigcup\limits_{\Y\in{{}^1\mathcal{P}}}{}^{\B}\mathcal{P}^{\prime}_{\Y},\;\;\; \mathcal{P}^{\prime}_{\Y}:=\bigcup\limits_{\B\in \mathcal{M}}{}^{\B}\mathcal{P}^{\prime}_{\Y}.
\end{eqnarray*}
Moreover, denote
\begin{eqnarray*}
\mathcal{P}^{\prime}:=\bigcup\limits_{\B\in \mathcal{M}}{^{\B}\mathcal{P}^{\prime}}=\bigcup\limits_{\Y\in{{}^1\mathcal{P}^{\prime}}}\mathcal{P}^{\prime}_{\Y}.
\end{eqnarray*}

According to \cite[Corollary 3.6]{Rad77}, since $H$ has the dual Chevalley property, the antipode $S$ of $H$ is bijective. Then for the mixed Hopf module $H_1/H_0$ in ${}^H\mathcal{M}_H$, we have $$  {}^{coH_0}(H_1/H_0)\otimes H_0\cong H_1/H_0,$$
where ${}^{coH_0}(H_1/H_0)$ is the left coinvariants of $H_0$ in $H_1/H_0$.
And the isomorphism maps $\overline{x}\otimes h$ to $\overline{x}\cdot h$, where $h\in H_0, \overline{x}\in  {}^{coH_0}(H_1/H_0)$.

The proof of the following lemma can be completed by the method analogous to that used in the proof of \cite[Remark 3.6, Corollary 3.9 and Theorem 3.10]{YLL24}.
\begin{lemma}\label{coro:BXcompleteprime}
\begin{itemize}
  \item[(1)]With the notations above, we have
  \begin{itemize}
    \item[(i)]the cardinal number $\mid{}^{\B}\mathcal{P}^{\prime}_{ \Y} \mid=u_{(\C, \B)};$
    \item[(ii)]the union $\mathcal{P}^{\prime}=\bigcup\limits_{\Y\in{{}^1\mathcal{P}^{\prime}}}\mathcal{P}^{\prime}_{\Y}$ is disjoint.
  \end{itemize}
  \item[(2)]
Let $C, D\in \mathcal{S}$ with basic multiplicative matrices $\C, \D\in\mathcal{M}$ respectively. Denote
$${}^{\C}\mathcal{P}^{\prime \D}:=\{\X^\prime\in \mathcal{P}^\prime\mid\X^\prime \text{ is a non-trivial }(\C, \D)\text{-primitive matrix}\}.$$
Then it is a complete family of non-trivial $(\C, \D)$-primitive matrices. Moreover, we have
$H_1/H_0=\bigoplus_{\X^{\prime}\in\mathcal{P}^{\prime}}\span(\overline{\X^{\prime}})$.
\end{itemize}
\end{lemma}
Now it is not difficult to verify the following lemma.
\begin{lemma}\label{lem:center}
Let $H$ be a non-cosemisimple Hopf algebra over $\k$ with the dual Chevalley property. If ${}^1\mathcal{S}=\{C_k\}$, then both $C_k$ and $S(C_k)$ are in the center of $\Bbb{Z}\mathcal{S}$.
\end{lemma}

\begin{proof}
For any $C_i\in\mathcal{S}$, suppose that
\begin{eqnarray*}
C_i\cdot C_k=\sum_{i\in I}\alpha_{ik}^t C_t,\\
C_k\cdot C_i=\sum_{i\in I}\alpha_{ki}^t C_t.
\end{eqnarray*}
Combining Lemma \ref{coro:BXcomplete} and \cite[Corollary 3.9]{YLL24}, we know that
$$\mid{}^{\C_i}{}\mathcal{P}^{\C_t}\mid=\mid{}^1\mathcal{P}^{\C_k}\mid\alpha_{ik}^t. $$
Using Lemma \ref{coro:BXcompleteprime}, one can show that
$$\mid{}^{\C_i}{}\mathcal{P}^{\prime\C_t}\mid=\mid{}^1\mathcal{P}^{\C_k}\mid\alpha_{ki}^t. $$
According to \cite[Corollary 2.11 and Lemma 2.17]{YLL24}, we have
\begin{eqnarray*}
&&\frac{1}{\sqrt{\dim_{\k}(C_i)}\sqrt{\dim_{\k}(C_t)}}\dim_{\k}((C_i\wedge C_t)/(C_i+C_t))\\
&=&\mid {}^{\C_i}\mathcal{P}^{\C_t}\mid\\
&=&\mid {}^{\C_i}\mathcal{P}^{\prime \C_t}\mid.
\end{eqnarray*}
It follows that $$\alpha_{ik}^t=\alpha_{ki}^t,$$
which means that $$C_i\cdot C_k=C_k\cdot C_i$$ for any $C_i\in \mathcal{S}.$ Besides, we have $$S(C_i)\cdot S(C_k)=S(C_k)\cdot S(C_i)$$ for any $C_i\in\mathcal{S}$. Thus both $C_k$ and $S(C_k)$ are in the center of $\Bbb{Z}\mathcal{S}$.
\end{proof}

\begin{lemma}\label{lem:connected}
Let $H$ be a non-cosemisimple Hopf algebra over $\k$ with the dual Chevalley property. Denote by $\mathcal{S}^\prime$ the set of simple subcoalgebras of $H_{(1)}$. Suppose that $\mid{}^1\mathcal{P}\mid=1$ and $C_k$ is the unique subcoalgebra contained in ${}^1\mathcal{S}$. Then $D\in\mathcal{S}^\prime$ if and only if there exist some $m, n\in\Bbb{Z}_+$ such that $(C_k)^m\cdot (S(C_k))^n$ contains $D$ with a nonzero coefficient.
\end{lemma}
\begin{proof}
By the same reason in the proof of \cite[Proposition 3.10]{YL24} and the fact that $C_k$ and $S(C_k)$ are in the center of $\Bbb{Z}\mathcal{S}$, we can easily prove the lemma.
\end{proof}
In the following part, suppose $H$ is a non-cosemisimple Hopf algebra over $\k$ with the dual Chevalley property of discrete corepresentation type such that the coradical of $H_{(1)}$ is infinite-dimensional. Suppose $\mid{}^1\mathcal{P}\mid=1$ and ${}^1\mathcal{S}=\{C_k\}$, where $\dim_{\k}(C_k)=4$.
From the fact that simple subcoalgebras are finite-dimensional and the proof of \cite[Proposition 5.5]{YLL24}, we can get an infinite sequence
\begin{eqnarray*}
S(C_k)\cdot C_k&=&\k1+\k g+D_1,\\
D_1\cdot S(C_k)&=&S(C_k)+D_2,\\
D_2\cdot C_k&=&D_1+D_3,\\
&\cdots&\\
D_{2i}\cdot C_k&=&D_{2i-1}+D_{2i+1},\\
D_{2i+1}\cdot S(C_k)&=&D_{2i}+D_{2i+2},\\
&\cdots&,
\end{eqnarray*}
where $g\in G(H)$ and $\dim_{\k}(D_i)=4$ for any $i\geq 1$. Moreover, we have $D_i\neq D_{i+2m}$ for any $i\geq 1, m>0.$ Otherwise, we can find a finite sub-quiver $\mathrm{Q}^\prime$ of $\mathrm{Q}(H)$ such that the separated quiver $\mathrm{Q}^\prime_s$ of $\mathrm{Q}^\prime$ is not a disjoint union of Dynkin diagrams, which is in contradiction with Lemma \ref{lem:Dynkin}.
Moreover, since $$S(C_k)\cdot C_k=\k1+\k g+D_1,$$ it follows that $$S(C_k)\cdot S^2(C_k)=\k1+\k S(g)+S(D_1).$$This means that $S(g)=g$ and $S(D_1)=D_1.$ Using Lemmas \ref{lemma:P^1=1^P} (4) and \ref{lem:center}, we have
$$ C_k\cdot\k g=C_k,$$ and $$C_k \cdot D_1=C_k+S(D_2).$$
It follows that $$C_k\cdot (C_k\cdot S(C_k))=3C_k+S(D_2).$$
For any $n\geq 2$, let $\mathcal{S}(n)$ be the set of all the $n^2$-dimensional simple subcoalgebras of $H$. Now we figure out $C_k\cdot C_k$.
\begin{lemma}
We have $$C_k\cdot C_k=\k h_1+\k h_2+ E$$ in $\Bbb{Z}\mathcal{S}$, where $h_1, h_2$ are two different group-like elements and $E$ is a $4$-dimensional simple subcoalgebra.
\end{lemma}

\begin{proof}
\begin{itemize}
  \item [(I)]
Suppose that $C_k\cdot C_k=E^{(4)}$, where $E^{(4)}\in\mathcal{S}(4)$.
We have $$E^{(4)}\cdot S(C_k)=C_k\cdot(C_k\cdot S(C_k))=3C_k+S(D_2). $$
According to Lemma \ref{lemma:P^1=1^P} (4), the separated quiver of $\mathrm{Q}(H_{(1)})$ contains a vertex which is the start vertex of $4$ arrows. Evidently, we can find a finite sub-quiver $\mathrm{Q}^\prime$ of $\mathrm{Q}(H)$ such that the separated quiver $\mathrm{Q}^\prime_s$ of $\mathrm{Q}^\prime$ is not a disjoint union of Dynkin diagrams. It follows from
Lemma \ref{lem:Dynkin} that $H$ is not of discrete corepresentation type.
\item [(II)]
Suppose $$C_k\cdot C_k=\k h+E^{(3)}, $$where $h\in G(H)$ and $E^{(3)}\in\mathcal{S}(3)$. A similar argument shows that $$\k h \cdot S(C_k)=C_k$$ and $$E^{(3)}\cdot S(C_k)=2C_k+S(D_2).$$
Combining Corollary \ref{coro:schurian} and Lemma \ref{lemma:P^1=1^P} (2), one can show that $H$ is not of discrete corepresentation type.\\
\item [(III)]Suppose $$C_k\cdot C_k=E_1^{(2)}+E_2^{(2)}, $$where $E_i^{(2)}\in\mathcal{S}(2)$ for $i=1, 2$. A similar argument shows that there exists some $i\in \{1, 2\}$ such that $$E_1^{(2)}\cdot S(C_k)=2C_k,$$ which follows that $H$ is not of discrete corepresentation type.
\item [(IV)]
Suppose $$C_k\cdot C_k=\k h_1+\k h_2+\k h_3+\k h_4, $$ where $h_1, h_2, h_3, h_4\in G(H)$. One can show that $H$ is not of discrete corepresentation type by the same taken.
\item [(V)]Suppose $$C_k\cdot C_k=\k h_1+\k h_2+ E$$ in $\Bbb{Z}\mathcal{S}$, where $h_1, h_2\in G(H)$ and $E\in\mathcal{S}(2)$. According to Corollary \ref{coro:schurian}, we know that $h_1\neq h_2$.
\end{itemize}
\end{proof}
Note that $S(D_1)=D_1$ and $$D_1\cdot S(C_k)=S(C_k)+D_2$$ in $\Bbb{Z}\mathcal{S}$. It follows from Lemma \ref{lemma:P^1=1^P} (4) that $$C_k\cdot D_1=C_k+S(D_2)$$ in $\Bbb{Z}\mathcal{S}$. Moreover, by the fact that $C_k, S(C_k)$ are in the center of $\Bbb{Z}\mathcal{S}$, we have
\begin{eqnarray*}
C_k\cdot(S(C_k)\cdot C_k)&=&C_k\cdot(\k 1+\k g+D_1)\\
&=&3C_k+S(D_2)\\
&=&(C_k\cdot C_k)\cdot S(C_k)\\
&=&(\k h_1+\k h_2+E)\cdot S(C_k).
\end{eqnarray*}
This means that $$E\cdot S(C_k)= C_k +S(D_2)$$ in $\Bbb{Z}\mathcal{S}$.

Recall that a quotient quiver $\overline{\mathrm{Q}}=(\overline{\mathrm{Q}}_0, \overline{\mathrm{Q}}_1)$ of $\mathrm{Q}=(\mathrm{Q}_0, \mathrm{Q}_1)$ is a quiver, whose vertices are blocks of partitions of $\mathrm{Q}_0$ and the number of arrows from $\overline{D}$ to $\overline{E}$ in $\overline{\mathrm{Q}}$ equals the total number of arrows from $D$ to $E$ for all $D\in\overline{D}$ and $E\in\overline{E}$. According to Lemma \ref{lemma:P^1=1^P} (4), the number of arrows from $C_i$ to $C_t$ is equal to the number of arrows from $S(C_t)$ to $S(C_i)$.
Now using Lemma \ref{lemma:P^1=1^P} (4), we can show that $\mathrm{Q}(H_{(1)})$ contains a sub-quiver which is the quotient quiver of the following form:
$$
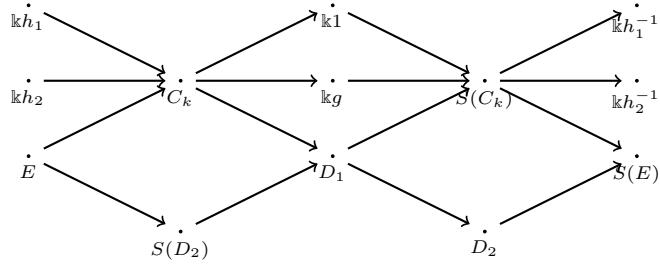
\begin{tikzpicture}
\filldraw [black] (4,0) circle (0.5pt) node[anchor=north]{${}_{\k h_2^{-1}}$};
\filldraw [black] (2,0) circle (0.5pt) node[anchor=north]{${}_{S(C_k)}$};
\filldraw [black] (0,0) circle (0.5pt) node[anchor=north]{${}_{\k g}$};
\filldraw [black] (-2,0) circle (0.5pt)node[anchor=north]{${}_{C_k}$};
\filldraw [black] (-4,0) circle (0.5pt)node[anchor=north]{${}_{\k h_2}$};
\filldraw [black] (4,1) circle (0.5pt) node[anchor=north]{${}_{\k h_1^{-1}}$};
\filldraw [black] (0,1) circle (0.5pt) node[anchor=north]{${}_{\k 1}$};
\filldraw [black] (-4,1) circle (0.5pt)node[anchor=north]{${}_{\k h_1}$};
\filldraw [black] (4,-1) circle (0.5pt) node[anchor=north]{${}_{S(E)}$};
\filldraw [black] (0,-1) circle (0.5pt) node[anchor=north]{${}_{D_1}$};
\filldraw [black] (-4,-1) circle (0.5pt)node[anchor=north]{${}_{E}$};
\filldraw [black] (2,-2) circle (0.5pt) node[anchor=north]{${}_{D_2}$};
\filldraw [black] (-2,-2) circle (0.5pt)node[anchor=north]{${}_{S(D_2)}$};
\draw[thick, ->] (-3.8,0).. controls (-3.8,0) and (-3.8,0) .. node[anchor=south]{}(-2.2,0);
\draw[thick, ->] (-1.8,0).. controls (-1.8,0) and (-1.8,0) .. node[anchor=south]{}(-0.2,0);
\draw[thick, ->] (0.2,0).. controls (0.2,0) and (0.2,0) .. node[anchor=south]{}(1.8,0);
\draw[thick, ->] (2.2,0).. controls (2.2,0) and (2.2,0) .. node[anchor=south]{}(3.8,0);
\draw[thick, ->] (-3.8,0.9).. controls (-3.8,0.9) and (-3.8,0.9) .. node[anchor=south]{}(-2.2,0.1);
\draw[thick, ->] (-1.8,0.1).. controls (-1.8,0.1) and (-1.8,0.1) .. node[anchor=south]{}(-0.2,0.9);
\draw[thick, ->] (0.2,0.9).. controls (0.2,0.9) and (0.2,0.9) .. node[anchor=south]{}(1.8,0.1);
\draw[thick, ->] (2.2,0.1).. controls (2.2,0.1) and (2.2,0.1) .. node[anchor=south]{}(3.8,0.9);
\draw[thick, ->] (-3.8,-0.9).. controls (-3.8,-0.9) and (-3.8,-0.9) .. node[anchor=south]{}(-2.2,-0.1);
\draw[thick, ->] (-1.8,-0.1).. controls (-1.8,-0.1) and (-1.8,-0.1) .. node[anchor=south]{}(-0.2,-0.9);
\draw[thick, ->] (0.2,-0.9).. controls (0.2,-0.9) and (0.2,-0.9) .. node[anchor=south]{}(1.8,-0.1);
\draw[thick, ->] (2.2,-0.1).. controls (2.2,-0.1) and (2.2,-0.1) .. node[anchor=south]{}(3.8,-0.9);
\draw[thick, ->] (-3.8,-1.1).. controls (-3.8,-1.1) and (-3.8,-1.1) .. node[anchor=south]{}(-2.2,-1.9);
\draw[thick, ->] (-1.8,-1.9).. controls (-1.8,-1.9) and (-1.8,-1.9) .. node[anchor=south]{}(-0.2,-1.1);
\draw[thick, ->] (0.2,-1.1).. controls (0.2,-1.1) and (0.2,-1.1) .. node[anchor=south]{}(1.8,-1.9);
\draw[thick, ->] (2.2,-1.9).. controls (2.2,-1.9) and (2.2,-1.9) .. node[anchor=south]{}(3.8,-1.1);
\end{tikzpicture}.
$$

Based on the consideration above, we can get the link quiver of $H_{(1)}$ when $C_k=S(C_k)$.
\begin{proposition}\label{coro:C=SC}
Let $H$ be a non-cosemisimple Hopf algebra over $\k$ with the dual Chevalley property of discrete corepresentation type such that the coradical of $H_{(1)}$ is infinite-dimensional. If $\mid{}^1\mathcal{P}\mid=1$ and ${}^1\mathcal{S}=\{C_k\}$, where $C_k=S(C_k)$ and $\dim_{\k}(C_k)=4$, then the link quiver $\mathrm{Q}(H_{(1)})$ of $H_{(1)}$ is of the following form:
\begin{eqnarray}\label{quiver4}
\begin{tikzpicture}
\filldraw [black] (4,0) circle (0.5pt) node[anchor=north]{$D_3$};
\filldraw [black] (-4,0) circle (0.5pt)node[anchor=north]{$\k1$};
\filldraw [black] (2,0) circle (0.5pt) node[anchor=north]{$D_2$};
\filldraw [black] (0,0) circle (0.5pt) node[anchor=north]{$D_1$};
\filldraw [black] (-2,0) circle (0.5pt)node[anchor=north]{$C_k$};
\filldraw [black] (6,0) circle (0.5pt) node[anchor=west]{$\cdots$};
\filldraw [black] (-2,1.4) circle (0.5pt) node[anchor=south]{$\k g$};
\draw[thick, ->] (-3.8,0.1).. controls (-3.8,0.1) and (-3.8,0.1) .. node[anchor=south]{}(-2.2,0.1) ;
\draw[thick, ->] (-2.2,-0.1).. controls (-2.2,-0.1) and (-2.2,-0.1) .. node[anchor=south]{}(-3.8,-0.1);
\draw[thick, ->] (-1.8,0.1).. controls (-1.8,0.1) and (-1.8,0.1) .. node[anchor=south]{}(-0.2,0.1) ;
\draw[thick, ->] (-0.2,-0.1).. controls (-0.2,-0.1) and (-0.2,-0.1) .. node[anchor=south]{}(-1.8,-0.1);
\draw[thick, ->] (0.2,0.1).. controls (0.2,0.1) and (0.2,0.1) .. node[anchor=south]{}(1.8,0.1) ;
\draw[thick, ->] (1.8,-0.1).. controls (1.8,-0.1) and (1.8,-0.1) .. node[anchor=south]{}(0.2,-0.1);
\draw[thick, ->] (2.2,0.1).. controls (2.2,0.1) and (2.2,0.1) .. node[anchor=south]{}(3.8,0.1) ;
\draw[thick, ->] (3.8,-0.1).. controls (3.8,-0.1) and (3.8,-0.1) .. node[anchor=south]{}(2.2,-0.1);
\draw[thick, ->] (4.2,0.1).. controls (4.2,0.1) and (4.2,0.1) .. node[anchor=south]{}(5.8,0.1) ;
\draw[thick, ->] (5.8,-0.1).. controls (5.8,-0.1) and (5.8,-0.1) .. node[anchor=south]{}(4.2,-0.1);
\draw[thick, ->] (-1.9,1.3).. controls (-1.9,1.3) and (-1.9,1.3) .. node[anchor=south]{}(-1.9,0.2) ;
\draw[thick, ->] (-2.1,0.2).. controls (-2.1,0.2) and (-2.1,0.2) .. node[anchor=south]{}(-2.1,1.3);
\end{tikzpicture},
\end{eqnarray}
where $g\in G(H)$ and $D_i$ are distinct $4$-dimensional simple subcoalgebras for any $i\geq 1$.
\end{proposition}
\begin{proof}
Using Lemma \ref{lem:connected}, we know that $(H_{(1)})_0$ is generated by $C_k$. If there exist some distinct $i, j\geq 0$ such that $D_i= D_j$, where $D_0=C_k$, then $(H_{(1)})_0$ is finite-dimensional. This is a contradiction. Now we claim that for any $i\geq 1$, we have $D_i=S(D_i)$. Indeed, suppose for any $i\leq n-1$, we have $D_i=S(D_i).$ When $i=n$, since $$D_{n-1}\cdot C_k=D_{n-2}+D_{n},$$
it follows that $$S(C_k)\cdot S(D_{n-1})=S(D_{n-2})+S(D_{n}).$$ According to Lemma \ref{lem:center} and the induction assumption, we know that $$D_{n-1}\cdot C_k=D_{n-2}+S(D_{n}),$$ which indicates that $D_n=S(D_n)$. By Lemma \ref{lemma:P^1=1^P} (4), we have thus proved the proposition.
\end{proof}

Finally, we consider the converse of Theorem \ref{thm:H10infinite}, that is, whether $H$ is of discrete corepresentation type if $H$ satisfies cases $(1), (2),$ or $(3)$.
\begin{remark}
Let $H$ be a non-cosemisimple Hopf algebra over $\k$ with the dual Chevalley property such that the coradical of $H_{(1)}$ is infinite-dimensional.
In case $(1)$ of Theorem \ref{thm:H10infinite}, it follows from Proposition \ref{prop:=Hinfi} that $H$ is of discrete corepresentation type. But in cases $(2)$ and $(3)$, $H$ is not necessarily a discrete corepresentation type Hopf algebra. In fact, $B^{m, n}(\lambda, s, t, k)$ is precisely the Hopf algebra corresponding to the quiver $\mathrm{Q}^{m, n}$. To ensure $B^{m, n}(\lambda, s, t, k)$ is of discrete corepresentation type, the authors in \cite{ISSZ24} imposed some restrictions on these parameters (see, for example, \cite[Theorem 5.15]{ISSZ24}). In other words, counterexamples exist in case $(2)$; $B^{m, m}(\lambda, s, t, k)$ being one such example, where $m\neq0.$ Besides, we will provide an example of non-discrete corepresentation type Hopf algebra satisfying case $(3)$ in the next section (see Remark \ref{rm:counterexample} below).
\end{remark}

\section{Example}\label{section4}
In this section, we will show that the situation (3) of Theorem \ref{thm:H10infinite} does occur. That is, we introduce a new algebra $H(e_{\pm 1}, f_{\pm 1}, u, v)$ and show that this algebra is a Hopf algebra with the dual Chevalley property of discrete corepresentation type such that its link quiver is of the form (\ref{quiver4}) in Proposition \ref{coro:C=SC}.
\begin{definition}\label{def:Hefuv}
As an algebra, $H(e_{\pm 1}, f_{\pm 1}, u, v)$ is generated by $u, v, e_{i}, f_{i}$ for $i\in\Bbb{Z}$, subject to the following relations
\begin{eqnarray*}
1=e_0+f_0, \;\;e_ie_j=e_{i+j},\;\;f_if_j=f_{i+j},\;\;e_if_j=f_je_i=0,
\end{eqnarray*}
\begin{eqnarray*}
e_iu=(-1)^iue_i,\;\;f_iu=(-1)^iuf_i,\;\;e_iv=(-1)^ive_i,\;\;f_iv=(-1)^ivf_i,
\end{eqnarray*}
\begin{eqnarray*}
u^2=v^2=0,\;\; uv=-vu,
\end{eqnarray*}
for any $i, j\in\Bbb{Z}$.

The comultiplication, counit and the antipode are given by
\begin{eqnarray*}
\Delta(e_i)=e_i\otimes e_i+f_i\otimes f_{-i},\;\;\varepsilon(e_i)=1,\;\;S(e_i)=e_{-i},
\end{eqnarray*}
\begin{eqnarray*}
\Delta(f_i)=e_i\otimes f_i+f_i\otimes e_{-i},\;\;\varepsilon(f_i)=0,\;\;S(f_i)=f_{i},
\end{eqnarray*}
\begin{eqnarray*}
\Delta(u)=1\otimes u+u\otimes e_1+v\otimes f_{-1},\;\;\varepsilon(u)=0,\;\;S(u)=-vf_{-1}-ue_{-1},
\end{eqnarray*}
\begin{eqnarray*}
\Delta(v)=1\otimes v+u\otimes f_1+v\otimes e_{-1},\;\;\varepsilon(v)=0,\;\;S(v)=-uf_{1}-ve_{1},
\end{eqnarray*}
for any $i\in\Bbb{Z}$.
\end{definition}
With operations defined above, we have
\begin{lemma}
 $H(e_{\pm 1}, f_{\pm 1}, u, v)$ is a Hopf algebra with the dual Chevalley property. Moreover, $H(e_{\pm1},f_{\pm 1},u,v)_0$ is spanned by $e_i,f_i, i \in \mathbb{Z}$ and $H(e_{\pm1},f_{\pm 1},u,v)=H(e_{\pm1},f_{\pm 1},u,v)_2$.
\end{lemma}
\begin{proof}
The proof is routine. For completeness and the convenience for other reader, we give the proof here. As usual, we decompose the proof into several steps.\\
$\bullet$ Step $1$ ($\Delta$ and $\varepsilon$ are algebra homomorphisms.) \\
First of all, it is clear that $\varepsilon$ is an algebra homomorphism. By definition, we have
\begin{eqnarray*}
\Delta(e_0+f_0)&=&e_0\otimes e_0+f_0\otimes f_{0}+e_0\otimes f_0+f_0\otimes e_{0}\\
&=&(e_0+f_0)\otimes(e_0+f_0)\\
&=&\Delta(1),
\end{eqnarray*}
and for any $i\in\Bbb{Z}$,
\begin{eqnarray*}
\Delta(e_i)\Delta(e_j)&=&(e_i\otimes e_i+f_i\otimes f_{-i})(e_j\otimes e_j+f_j\otimes f_{-j})\\
&=&e_{i+j}\otimes e_{i+j}+f_{i+j}\otimes f_{-i-j}\\
&=&\Delta(e_{i+j}),\\
\Delta(f_i)\Delta(f_j)&=&(e_i\otimes f_i+f_i\otimes e_{-i})(e_j\otimes f_j+f_j\otimes e_{-j})\\
&=&e_{i+j}\otimes f_{i+j}+f_{i+j}\otimes e_{-i-j}\\
&=&\Delta(f_{i+j}),\\
\Delta(e_i)\Delta(f_j)&=&(e_i\otimes e_i+f_i\otimes f_{-i})(e_j\otimes f_j+f_j\otimes e_{-j})\\
&=&0,\\
\Delta(f_j)\Delta(e_i)&=&(e_j\otimes f_j+f_j\otimes e_{-j})(e_i\otimes e_i+f_i\otimes f_{-i})\\
&=&0.
\end{eqnarray*}
Meanwhile, for any $i\in\Bbb{Z}$,
\begin{eqnarray*}
\Delta(e_i)\Delta(u)&=&(e_i\otimes e_i+f_i\otimes f_{-i})(1\otimes u+u\otimes e_1+v\otimes f_{-1})\\
&=&e_i\otimes e_i u+f_i\otimes f_{-i}u+e_i u\otimes e_{i+1}+f_{i}v\otimes f_{-i-1}\\
&=&(-1)^i(e_i\otimes u e_i +f_i\otimes uf_{-i}+ue_i \otimes e_{i+1}+vf_{i}\otimes f_{-i-1})\\
&=&(-1)^i\Delta(u)\Delta(e_i),
\end{eqnarray*}
and
\begin{eqnarray*}
\Delta(f_i)\Delta(u)&=&(e_i\otimes f_i+f_i\otimes e_{-i})(1\otimes u+u\otimes e_1+v\otimes f_{-1})\\
&=&e_i\otimes f_i u+f_{i}\otimes e_{-i}u+f_iu\otimes e_{-i+1}+e_iv\otimes f_{i-1}\\
&=&(-1)^i(e_i\otimes u f_i +f_{i}\otimes u e_{-i}+uf_i\otimes e_{-i+1}+ve_i\otimes f_{i-1})\\
&=&(-1)^i\Delta(u)\Delta(f_i).
\end{eqnarray*}
Using a similar argument, one can get $\Delta(e_i)\Delta(v)=(-1)^i\Delta(v)\Delta(e_i)$ and $\Delta(f_i)\Delta(v)=(-1)^i\Delta(v)\Delta(f_i).$
Moreover, we find that
\begin{eqnarray*}
\Delta(u)\Delta(v)&=&(1\otimes u+u\otimes e_1+v\otimes f_{-1})(1\otimes v+u\otimes f_1+v\otimes e_{-1})\\
&=&1\otimes uv+u\otimes uf_1+v\otimes u e_{-1}+u\otimes e_1 v+uv\otimes e_0+v\otimes f_{-1}v+vu\otimes f_0
\end{eqnarray*}
and
\begin{eqnarray*}
\Delta(v)\Delta(u)&=&(1\otimes v+u\otimes f_1+v\otimes e_{-1})(1\otimes u+u\otimes e_1+v\otimes f_{-1})\\
&=&1\otimes vu+u\otimes ve_1+ v\otimes vf_{-1}+u\otimes f_1u+uv\otimes f_0+v\otimes e_{-1}u+vu\otimes e_{0}.
\end{eqnarray*}
It follows that $\Delta(u)\Delta(v)=-\Delta(v)\Delta(u)$.
Through direct calculation, we have
\begin{eqnarray*}
\Delta(u)\Delta(u)&=&(1\otimes u+u\otimes e_1+v\otimes f_{-1})(1\otimes u+u\otimes e_1+v\otimes f_{-1})\\
&=&u\otimes ue_1+v\otimes uf_{-1}+u\otimes e_{1}u+v\otimes f_{-1}u\\
&=&0
\end{eqnarray*}
and
\begin{eqnarray*}
\Delta(v)\Delta(v)&=&(1\otimes v+u\otimes f_1+v\otimes e_{-1})(1\otimes v+u\otimes f_1+v\otimes e_{-1})\\
&=&u\otimes vf_1+v\otimes ve_{-1}+u\otimes f_1 v+v\otimes e_{-1}v\\
&=&0.
\end{eqnarray*}\\
$\bullet$ Step $2$ (Coassociative and counit.) \\
  Indeed, for any $i\in\Bbb{Z}$,
  \begin{eqnarray*}
  (\Delta\otimes \id)\Delta(e_i)&=&(\Delta\otimes \id)(e_i\otimes e_i+f_i\otimes f_{-i})\\
  &=&e_i\otimes e_i\otimes e_i+f_i\otimes f_{-i}\otimes e_i+e_i\otimes f_i\otimes f_{-i}+f_i\otimes e_{-i}\otimes f_{-i}
  \end{eqnarray*}
  and
  \begin{eqnarray*}
  (\id\otimes \Delta)\Delta(e_i)&=&(\id\otimes \Delta)(e_i\otimes e_i+f_i\otimes f_{-i})\\
  &=&e_i\otimes e_i\otimes e_i+e_i\otimes f_i\otimes f_{-i}+f_i\otimes e_{-i}\otimes f_{-i}+f_i\otimes f_{-i}\otimes e_i.
  \end{eqnarray*}
  It is not hard to see that they are the same and thus $(\Delta\otimes \id)\Delta(e_i)=(\id\otimes \Delta)\Delta(e_i).$
  Similarly, one can show that both $(\Delta\otimes \id)\Delta(f_i)$ and $(\id\otimes \Delta)\Delta(f_i)$ equal to
  $$e_i\otimes e_i\otimes f_i+f_i\otimes f_{-i}\otimes f_i+e_i\otimes f_i\otimes e_{-i}+f_i\otimes e_{-i}\otimes e_{-i},$$
  for any $i\in\Bbb{Z}$.
  Moreover, a simple computation shows that
  \begin{eqnarray*}
  (\Delta\otimes \id)\Delta(u)&=&(\id\otimes \Delta)\Delta(u)\\
  &=&(1\otimes 1\otimes u+1\otimes u\otimes e_1+u\otimes e_1\otimes e_1+v\otimes f_{-1}\otimes e_1\\
  &&+1\otimes v\otimes f_{-1}+u\otimes f_1\otimes f_{-1}+v\otimes e_{-1}\otimes f_{-1})
   \end{eqnarray*}
   and
  \begin{eqnarray*}
 (\Delta\otimes \id)\Delta(v)&=& (\id\otimes \Delta)\Delta(v)\\
  &=&(1\otimes 1\otimes v+1\otimes u\otimes f_1+u\otimes e_1\otimes f_1+v\otimes f_{-1}\otimes f_1\\
  &&+1\otimes v\otimes e_{-1}+u\otimes f_1\otimes e_{-1}+v\otimes e_{-1}\otimes e_{-1}).
   \end{eqnarray*}
  The verification of the axiom for counit is easy and it is omitted.\\
  $\bullet$ Step $3$ (Antipode is an algebra anti-homomorphism.)\\
  It is clear that
  \begin{eqnarray*}
  S(e_0+f_0)=e_0+f_0=1,
    \end{eqnarray*}
     \begin{eqnarray*}
  S(e_j)S(e_i)=e_{-j}e_{-i}=e_{-i-j}=S(e_{i+j}),\;\;S(f_j)S(f_i)=f_jf_i=f_{i+j}=S(f_{i+j}),
  \end{eqnarray*}
  \begin{eqnarray*}
  S(f_j)S(e_i)=f_je_{-i}=0,\;\;S(e_i)S(f_j)=e_{-i}f_{j}=0.
   \end{eqnarray*}
   We also have
   \begin{eqnarray*}
   S(u)S(e_i)=(-vf_{-1}-ue_{-1})e_{-i}=-ue_{-i-1}
    \end{eqnarray*}
    and
    \begin{eqnarray*}
    S(e_i)S(u)=e_{-i}(-vf_{-1}-ue_{-1})=(-1)^{i+1}ue_{-i-1},
    \end{eqnarray*}
    which follows that $S(u)S(e_i)=(-1)^iS(e_i)S(u).$ Besides,
    \begin{eqnarray*}
    &&S(u)S(f_i)=-vf_{i-1}=(-1)^iS(f_i)S(u), \\
    &&S(v)S(e_i)=-ve_{1-i}=(-1)^iS(e_i)S(v),\\
    &&S(v)S(f_i)=-uf_{i+1}=(-1)^iS(f_i)S(v),\\
    &&S(u)S(u)=(-vf_{-1}-ue_{-1})(-vf_{-1}-ue_{-1})=0,\\
   && S(v)S(v)=(-uf_{1}-ve_{1})(-uf_{1}-ve_{1})=0,\\
   &&S(v)S(u)=-uvf_0-vue_0=-S(u)S(v).
    \end{eqnarray*}
  $\bullet$ Step $4$ (The axiom for antipode.) \\
  By definition, we have
  \begin{eqnarray*}
  &&e_iS(e_i)+f_iS(f_{-i})=e_0+f_0=1=\varepsilon(e_i),\\
  &&S(e_i)e_i+s(f_i)f_{-i}=e_0+f_0=1=\varepsilon(e_i),\\
  &&e_iS(f_i)+f_iS(e_{-i})=0=\varepsilon(f_i),\\
  &&S(e_i)f_i+S(f_i)e_{-i}=0=\varepsilon(f_i),\\
  &&S(u)+uS(e_1)+vS(f_{-1})=0=\varepsilon(u),\\
  &&u+S(u)e_1+S(v)f_{-1}=0=\varepsilon(u),\\
  &&S(v)+uS(f_1)+vS(e_{-1})=0=\varepsilon(v),\\
  &&v+S(u)f_1+S(v)e_{-1}=0=\varepsilon(v).
  \end{eqnarray*}
  By steps $1, 2 ,3, 4$,  $H(e_{\pm 1}, f_{\pm 1}, u, v)$ is a Hopf algebra.

  Denote $H(e_{\pm1},f_{\pm 1})=\operatorname{span}\{e_i, f_i\mid i\in \mathbb{Z}\}$. We know that $$H(e_{\pm1},f_{\pm 1})=\Bbbk1\oplus\Bbbk g\oplus(\bigoplus_{i\geq 1} \span\{e_i, f_i, e_{-i}, f_{-i}\}),$$
where $g=e_0-f_0$. This means that $H(e_{\pm1},f_{\pm 1})$ is a cosemisimple Hopf subalgebra of $H(e_{\pm1},f_{\pm 1}, u, v).$
Denote $$H(e_{\pm1},f_{\pm 1},u,v)(0)=H(e_{\pm1},f_{\pm 1}),$$ $$H(e_{\pm1},f_{\pm 1}, u,v)(1)= H(e_{\pm1},f_{\pm 1})u\oplus H(e_{\pm1},f_{\pm 1})v$$ and $$H(e_{\pm1},f_{\pm 1}, u,v)(2)=H(e_{\pm1},f_{\pm 1})uv.$$
We have $$H(e_{\pm1},f_{\pm 1},u,v)=H(e_{\pm1},f_{\pm 1}, u, v)(0)\oplus H(e_{\pm1},f_{\pm 1}, u, v)(1)\oplus H(e_{\pm1},f_{\pm 1}, u, v)(2).$$
It is straightforward to show that
$H(e_{\pm1},f_{\pm 1},u,v)$ is a graded algebra with the grading defined as above and $S(H(e_{\pm1},f_{\pm 1},u,v)(j))\subseteq H(e_{\pm1},f_{\pm 1},u,v)(j)$ for $j=0, 1, 2$.
Note that
\begin{eqnarray*}
    \Delta (u)&=&1\otimes u+u\otimes e_1+ v\otimes f_{-1}\\
&\in& H(e_{\pm1},f_{\pm 1}, u, v)(0)\otimes H(e_{\pm1},f_{\pm 1}, u, v)(1)\\
&&+H(e_{\pm1},f_{\pm 1}, u, v)(1)\otimes H(e_{\pm1},f_{\pm 1}, u, v)(0).
\end{eqnarray*}
It follows that
\begin{eqnarray*}
\Delta(H(e_{\pm1},f_{\pm 1})u)&=& \Delta(H(e_{\pm1},f_{\pm 1}))\Delta(u)\\
&\subseteq& H(e_{\pm1},f_{\pm 1}, u, v)(0)\otimes H(e_{\pm1},f_{\pm 1}, u, v)(1)\\
&&+H(e_{\pm1},f_{\pm 1}, u, v)(1)\otimes H(e_{\pm1},f_{\pm 1}, u, v)(0).
\end{eqnarray*}
A similar argument shows that
\begin{eqnarray*}
\Delta(H(e_{\pm1},f_{\pm 1})v)
&\subseteq& H(e_{\pm1},f_{\pm 1}, u, v)(0)\otimes H(e_{\pm1},f_{\pm 1}, u, v)(1)\\
&&+H(e_{\pm1},f_{\pm 1}, u, v)(1)\otimes H(e_{\pm1},f_{\pm 1}, u, v)(0).
\end{eqnarray*}
Moreover,
\begin{eqnarray*}
\Delta(H(e_{\pm1},f_{\pm 1})uv)&=&\Delta(H(e_{\pm1},f_{\pm 1}))\Delta(u)\Delta(v)\\
&\subseteq& H(e_{\pm1},f_{\pm 1})(0)\otimes H(e_{\pm1},f_{\pm 1})(2)\\&+& H(e_{\pm1},f_{\pm 1})(1)\otimes H(e_{\pm1},f_{\pm 1})(1)\\&+&H(e_{\pm1},f_{\pm 1})(2)\otimes H(e_{\pm1},f_{\pm 1})(0).
\end{eqnarray*}
Thus we can show that $H(e_{\pm1},f_{\pm 1},u,v)$ is a graded Hopf algebra.
Moreover, we have $$H(e_{\pm1},f_{\pm 1}, u,v)(0)\subseteq \bigoplus\limits_{j=0}^1 H(e_{\pm1},f_{\pm 1}, u,v)(j)\subseteq \bigoplus\limits_{j=0}^2H(e_{\pm1},f_{\pm 1},u,v)(j)$$
is a Hopf algebra filtration. It follows from \cite[Proposition 11.1.1]{Swe69} that $H(e_{\pm1},f_{\pm 1},u,v)_0\subseteq H(e_{\pm1},f_{\pm 1},u,v)(0).$ By the fact that $H(e_{\pm1},f_{\pm 1})$ is a cosemisimple Hopf algebra, we know that $H(e_{\pm1},f_{\pm 1},u,v)_0=H(e_{\pm1},f_{\pm 1})$ and $H(e_{\pm1},f_{\pm 1},u,v)=H(e_{\pm1},f_{\pm 1},u,v)_2$. This means that $H(e_{\pm1},f_{\pm 1},u,v)$ has the dual Chevalley property.
\end{proof}
In the following part, denote $g=e_0-f_0$, it is clear that $g$ is a group-like element of order $2$. For any $i\geq 1$, denote
$C_i=\span\{e_i, f_i, e_{-i}, f_{-i}\}$. We can show that each $C_i$ is a simple subcoalgebra with basic multiplicative matrix $\C_i$, where
$$
\C_i=\left(\begin{array}{cc}
e_i&f_i\\
f_{-i}&e_{-i}
 \end{array}\right).
$$
It should be pointed out that $C_i=S(C_i)$ for any $i\geq 1.$
We can obtain $$H(e_{\pm 1}, f_{\pm 1}, u, v)_0=\k1\oplus \k g\oplus \bigoplus_{i\geq 1}C_i,$$ which follows that $\mathcal{S}=\{\k1, \k g\}\cup\{ C_i\mid  i\geq 1\}$.

From the definition, it follows automatically that
$\left(\begin{array}{cc}
u&v
 \end{array}\right)$ is a non-trivial $(1, \C_1)$-primitive matrix.
Besides, we have $$C_1\cdot C_1=\k 1+\k g+C_2$$ and $$C_i\cdot C_1=C_{i+1}+C_{i-1}\;\;\text{for }i\geq2 $$ in $\Bbb{Z}\mathcal{S}.$
We can show that
\begin{eqnarray*}
{}^{1}\mathcal{P}^{\C_1}&=&
\{\left(\begin{array}{cc}
u&v
 \end{array}\right)\},\\
 {}^{g}\mathcal{P}^{\C_1}&=&
 \{\left(\begin{array}{cc}
-gu&gv
 \end{array}\right)\},\\
{}^{\C_1} \mathcal{P}^{1}&=&
\{\left(\begin{array}{cc}
-uf_1-ve_1\\
-vf_{-1}-ue_{-1}
 \end{array}\right)\},\\
 {}^{\C_1} \mathcal{P}^{g}&=&
\{\left(\begin{array}{cc}
uf_1-ve_1\\
ue_{-1}-vf_{-1}
 \end{array}\right)\}.
\end{eqnarray*}
For any $i\geq 1$, we know that
\begin{eqnarray*}
 {}^{\C_i} \mathcal{P}^{\C_{i+1}}&=&
\{\left(\begin{array}{cc}
e_iu& f_i v\\
f_{-i}u&e_{-i}v
 \end{array}\right)\},\\
 {}^{\C_{i+1}} \mathcal{P}^{\C_{i}}&=&
\{\left(\begin{array}{cc}
e_{i+1}v& f_{i+1} u\\
f_{-i-1}v&e_{-i-1}u
 \end{array}\right)\}.
\end{eqnarray*}
By the construction of $\mathcal{P}$ in subsection \ref{subsection1.2}, it turns out that $$\mathcal{P}={}^{1}\mathcal{P}^{\C_1}\cup{}^{g}\mathcal{P}^{\C_1}\cup{}^{\C_1} \mathcal{P}^{1}\cup{}^{\C_1} \mathcal{P}^{g}\cup(\bigcup_{i\geq 1}{}^{\C_i} \mathcal{P}^{\C_{i+1}})\cup(\bigcup_{i\geq 1}{}^{\C_{i+1}} \mathcal{P}^{\C_{i}}).$$

As a consequence, the link quiver of $H(e_{\pm 1}, f_{\pm 1}, u, v)$ is of the following form:
\begin{eqnarray*}
\begin{tikzpicture}
\filldraw [black] (4,0) circle (0.5pt) node[anchor=north]{$C_4$};
\filldraw [black] (-4,0) circle (0.5pt)node[anchor=north]{$\k1$};
\filldraw [black] (2,0) circle (0.5pt) node[anchor=north]{$C_3$};
\filldraw [black] (0,0) circle (0.5pt) node[anchor=north]{$C_2$};
\filldraw [black] (-2,0) circle (0.5pt)node[anchor=north]{$C_1$};
\filldraw [black] (6,0) circle (0.5pt) node[anchor=west]{$\cdots$};
\filldraw [black] (-2,1.4) circle (0.5pt) node[anchor=south]{$\k g$};
\draw[thick, ->] (-3.8,0.1).. controls (-3.8,0.1) and (-3.8,0.1) .. node[anchor=south]{}(-2.2,0.1) ;
\draw[thick, ->] (-2.2,-0.1).. controls (-2.2,-0.1) and (-2.2,-0.1) .. node[anchor=south]{}(-3.8,-0.1);
\draw[thick, ->] (-1.8,0.1).. controls (-1.8,0.1) and (-1.8,0.1) .. node[anchor=south]{}(-0.2,0.1) ;
\draw[thick, ->] (-0.2,-0.1).. controls (-0.2,-0.1) and (-0.2,-0.1) .. node[anchor=south]{}(-1.8,-0.1);
\draw[thick, ->] (0.2,0.1).. controls (0.2,0.1) and (0.2,0.1) .. node[anchor=south]{}(1.8,0.1) ;
\draw[thick, ->] (1.8,-0.1).. controls (1.8,-0.1) and (1.8,-0.1) .. node[anchor=south]{}(0.2,-0.1);
\draw[thick, ->] (2.2,0.1).. controls (2.2,0.1) and (2.2,0.1) .. node[anchor=south]{}(3.8,0.1) ;
\draw[thick, ->] (3.8,-0.1).. controls (3.8,-0.1) and (3.8,-0.1) .. node[anchor=south]{}(2.2,-0.1);
\draw[thick, ->] (4.2,0.1).. controls (4.2,0.1) and (4.2,0.1) .. node[anchor=south]{}(5.8,0.1) ;
\draw[thick, ->] (5.8,-0.1).. controls (5.8,-0.1) and (5.8,-0.1) .. node[anchor=south]{}(4.2,-0.1);
\draw[thick, ->] (-1.9,1.3).. controls (-1.9,1.3) and (-1.9,1.3) .. node[anchor=south]{}(-1.9,0.2) ;
\draw[thick, ->] (-2.1,0.2).. controls (-2.1,0.2) and (-2.1,0.2) .. node[anchor=south]{}(-2.1,1.3);
\end{tikzpicture}.
\end{eqnarray*}

We now turn to the category of finite-dimensional right comodules over $H(e_{\pm 1}, f_{\pm 1}, u, v)$. For convenience, for any matrix $\A:=(a_{ij})_{m\times n}$ and $\B:=(b_{ij})_{n\times l}$ over $H(e_{\pm 1}, f_{\pm 1}, u, v)$, denote the following matrix
    $$\A\;\widetilde{\otimes}\;\B:=\left(\sum\limits_{k=1}^n a_{ik}\otimes b_{kl}\right)_{m\times l}.$$

Let $C, D, E, F\in\mathcal{S}$ with $\dim_{\Bbbk}(C)=r^2, \dim_{\Bbbk}(D)=s^2, \dim_{\Bbbk}(E)=t^2, \dim_{\Bbbk}(F)=u^2$ respectively, where $r, s, t, u\in\{1,2\}$. Suppose $\X_{r\times s}=\left(x_{ij}\right)_{r\times s}\in\mathcal{P}$ is a non-trivial $(\C, \D)$-primitive matrix, $\Y_{r\times t}=\left(y_{ij}\right)_{r\times t}\in\mathcal{P}$ is a non-trivial $(\C, \E)$-primitive matrix and $\Z_{u\times s}=\left(z_{ij}\right)_{u\times s}\in\mathcal{P}$ is a non-trivial $(\F, \D)$-primitive matrix.

It is clear that $$S=\span\{c_{11}, c_{12}, \cdots, c_{1r}\}$$ is a simple right $H(e_{\pm 1}, f_{\pm 1}, u, v)$-comodule with
$$
\rho(S)=S\;\widetilde{\otimes} \;\C.
$$

Moreover,
$$U=\span\{c_{11}, c_{12}, \cdots, c_{1r}, x_{11}, x_{12}, \cdots, x_{1s}\}$$ is a indecomposable right $H(e_{\pm 1}, f_{\pm 1}, u, v)$-comodule with
$$
\rho(U)=U\;\widetilde{\otimes} \;
\left(\begin{array}{cc}
\C & \X\\
0&  \D
 \end{array}\right).
$$

We also have $$V=\span\{c_{11}, c_{12}, \cdots, c_{1r}, x_{11}, x_{12}, \cdots, x_{1s}, y_{11}, y_{12}, \cdots, y_{1t}\}$$is a indecomposable right $H(e_{\pm 1}, f_{\pm 1}, u, v)$-comodule with
$$
\rho(V)=V\;\widetilde{\otimes}\;
\left(\begin{array}{ccc}
\C & \X&\Y\\
&  \D\\
&&\E
 \end{array}\right).
$$

For any non-zero $k\in\k$, denote$$W(k)=\span\{c_{11}, c_{12}, \cdots, c_{1r}, kf_{11}, kf_{12}, \cdots, kf_{1u}, x_{11}+kz_{11}, x_{12}+kz_{12}, \cdots, x_{1s}+kz_{1s}\}.$$
One can easily show that $W(k)$ is a indecomposable right $H(e_{\pm 1}, f_{\pm 1}, u, v)$-comodule with
$$
\rho(W(k))=W(k)\;\widetilde{\otimes}\;
\left(\begin{array}{ccc}
\C & &\X\\
&  \F&\Z\\
&&\D
 \end{array}\right).
$$
Besides, we know that $W(k)\cong W(l)$ as right $H(e_{\pm 1}, f_{\pm 1}, u, v)$-comodule for any non-zero $k, l\in\k.$
From this discussion, we get some indecomposable right $H(e_{\pm 1}, f_{\pm 1}, u, v)$-comodules of small dimension.

In order to show that $H(e_{\pm 1}, f_{\pm 1}, u, v)$ is of discrete corepresentation type, we consider a special case of subcoalgebras of $H(e_{\pm 1}, f_{\pm 1}, u, v)$.
For $N\geq 1$, let $$H^N(e_{\pm 1}, f_{\pm 1}, u, v)=\span\{e_i, f_i, e_iu, e_i v, f_iu, f_i v, e_i uv, f_i uv\mid -N\leq i\leq N\}\oplus C_{N+1}.$$ We have the following lemma.
\begin{lemma}\label{lem:HNfinite}
$H^N(e_{\pm 1}, f_{\pm 1}, u, v)$ is of finite corepresentation type.
\end{lemma}
\begin{proof}
Since $H(e_{\pm 1}, f_{\pm 1}, u, v)=H(e_{\pm 1}, f_{\pm 1}, u, v)_2,$ it follows that $H^N(e_{\pm 1}, f_{\pm 1}, u, v)$ only contains comodules with Loewy length at most 3. Note that indecomposable comodules with Loewy length $1$ are simple comodules. This means that the number of non-isomorphic indecomposable comodules with Loewy length $1$ is finite.\\
Denote
\begin{eqnarray*}
M&=&\span\{1, u, v\}+\span\{g, -gu, gv\}+\span\{e_1, f_1, -uf_1-ve_1\}+\span\{e_1, f_1, uf_1-ve_1\}\\
&&+(\sum_{1\leq i\leq N}\span\{e_i, f_i, e_iu, f_i v\})+(\sum_{1\leq i\leq N}\span\{e_{i}, f_{i}, e_{i}v, f_{i}u\})\\
&&+(\sum_{3\leq i\leq N-1}\span\{e_i, f_i, e_{i-2}, f_{i-2}, e_i v+e_{i-2}u, f_i u+ f_{i-2}v\}).
\end{eqnarray*}
Up to isomorphism, any indecomposable comodule $N$ with Loewy length $2$ is a subcomodule of $M$
such that the link-quiver of coefficient coalgebra $\operatorname{cf}(N)$ is connected. That is, $\operatorname{cf}(N)$ is a subcoalgebra of $H^N(e_{\pm 1}, f_{\pm 1}, u, v)_1$ and the link-quiver $\mathrm{Q}(\operatorname{cf}(N))$ of $\operatorname{cf}(N)$ is a connected sub-quiver of $\mathrm{Q}(H^N(e_{\pm 1}, f_{\pm 1}, u, v))$. Thus there are only finitely many non-isomorphic indecomposable comodules with Loewy length $2$.\\
Denote
\begin{eqnarray*}
V_0&=&\span\{1, u, v, uv\}\\
V_{1}&=&\span\{g, gu, gv, guv\}\\
V_{i+1}&=&\span\{e_i, f_i, e_i u, f_i v, e_i v, f_i u, e_i uv, f_i uv\}) \;\;\text{for any}\;i\geq 1.
\end{eqnarray*}
For any indecomposable comodule $U$ with Loewy length $3$, $U$ is an extension of $Soc(U)$, which is a direct sum of simple comodules, and $U/ Soc(U)$, which has Loewy length $2$.
This leads to any indecomposable comodule $U$ with Loewy length $3$ is isomorphic to $V_i$ for some $i\geq 0$. Hence there are only finitely many non-isomorphic indecomposable comodules with Loewy length $3$.\\
In conclusion, $H^N(e_{\pm 1}, f_{\pm 1}, u, v)$ is of finite corepresentation type.
\end{proof}
We conclude this section by point out that $H(e_{\pm 1}, f_{\pm 1}, u, v)$ is of discrete corepresentation type.
\begin{proposition}
$H(e_{\pm 1}, f_{\pm 1}, u, v)$ is a Hopf algebra with the dual Chevalley property of discrete corepresentation type.
\end{proposition}
\begin{proof}
Indeed, any finite-dimensional subcoalgebra of $H(e_{\pm 1}, f_{\pm 1}, u, v)$ is a subcoalgebra of $H^N(e_{\pm 1}, f_{\pm 1}, u, v)$ for some $N$. Using Lemma \ref{lem:HNfinite}, we know that any finite-dimensional subcoalgebra of $H(e_{\pm 1}, f_{\pm 1}, u, v)$ is of finite corepresentation type.
Now for any finite dimension vector $\underline{d}$, denote by $\operatorname{cf}(\underline{d})$ the smallest subcoalgebra of $H^N(e_{\pm 1}, f_{\pm 1}, u, v)$ such that all the right $H$-comodules of dimension vector $\underline{d}$ have their coefficient coalgebra contained in $\operatorname{cf}(\underline{d})$.
Using \cite[Lemma 2.6]{Iov18b}, we can show that $\operatorname{cf}(\underline{d})$ is finite-dimensional, which follows that $\operatorname{cf}(\underline{d})$ is of finite corepresentation type. Therefore, the set of isomorphism classes of dimension vector $\underline{d}$ is finite. This implies that $H(e_{\pm 1}, f_{\pm 1}, u, v)$ is of discrete corepresentation type.
\end{proof}
\begin{remark}\label{rm:counterexample}
If we remove $uv=-vu$ from Definition \ref{def:Hefuv} and keep the rest of the Hopf algebra structure unchanged, we obtain a new Hopf algebra $H^\prime(e_{\pm 1}, f_{\pm 1}, u, v)$. Since $$H(e_{\pm 1}, f_{\pm 1}, u, v)_1=H^\prime(e_{\pm 1}, f_{\pm 1}, u, v)_1,$$ $H(e_{\pm 1}, f_{\pm 1}, u, v)$ and $H^\prime(e_{\pm 1}, f_{\pm 1}, u, v)$ have the same link quiver with each other. However, their corepresentation type are different. Indeed, for any $\lambda\in\k, i\in\Bbb{Z}$, denote
$$V_i(\lambda)=\span\{e_i, f_i, e_iu, f_iv, e_iv, f_iu, e_iuv+\lambda e_ivu, f_ivu+\lambda f_iuv\}.$$
We know that
\begin{eqnarray*}
&&\rho((e_i, f_i, e_iu, f_iv, e_iv, f_iu, e_iuv+\lambda e_ivu, f_ivu+\lambda f_iuv))\\
&=&(e_i, f_i, e_iu, f_iv, e_iv, f_iu, e_iuv+\lambda e_ivu, f_ivu+\lambda f_iuv)\widetilde{\otimes}\\
&&
\left(\begin{array}{cccccccc}
e_i&f_i&e_iu&f_iv&e_iv&f_iu&e_iuv+\lambda e_ivu&f_ivu+\lambda f_iuv\\
f_{-i}&e_{-i}&f_{-i}u&e_{-i}v&f_{-i}v&e_{-i}u& f_{-i}uv+\lambda f_{-i}vu& e_{-i}vu+\lambda e_{-i}uv\\
0&0&e_{i+1}&f_{i+1}&0&0&(1-\lambda)e_{i+1} v&(1-\lambda)f_{i+1}u\\
0&0& f_{-i-1}&e_{-i-1}&0&0&(1-\lambda)f_{-i-1}v&(1-\lambda)e_{-i-1}u\\
0&0&0&0&e_{i-1}&f_{i-1}&(\lambda-1)e_{i-1}u&(\lambda-1)f_{i-1}v\\
0&0&0&0&f_{-i+1}&e_{-i+1}&(\lambda-1)f_{-i+1}u&(\lambda-1)e_{-i+1}v\\
0&0&0&0&0&0&e_i&f_i\\
0&0&0&0&0&0&f_{-i}&e_{-i}
 \end{array}\right).
\end{eqnarray*}
Note that $\{V_i(\lambda)\}_{\lambda\in\k}$ is a family of non-isomorphic indecomposable right $H^\prime(e_{\pm 1}, f_{\pm 1}, u, v)$-comodules admitting the same dimension vector. Thus $H^\prime(e_{\pm 1}, f_{\pm 1}, u, v)$ is not of discrete corepresentation type.
\end{remark}

\section*{Acknowledgement}
The authors highly appreciate the referee's detailed comments, which have been invaluable in improving this paper. The second author was supported by National Key R$\&$D Program of China 2024YFA1013802 and NSFC 12271243.
\section*{}
\subsection*{Author Contributions}
All authors contributed to all aspects of this project.

\subsection*{Availability of data and materials}
Not applicable.
\section*{Declarations}
\subsection*{Competing interests}
The authors declare no competing interests.
\subsection*{Ethical Approval} Not applicable.

\end{document}